\def\ART{TF}

\if\ART
\documentclass[dvips,a4paper]{article}
\else
\documentclass[dvips,a4paper]{amsart}
\fi
\usepackage{amsmath,amssymb, amsfonts}
\if\ART
\usepackage{amsthm}
\fi
\usepackage[dvips]{graphicx}
\usepackage{tikz}
\theoremstyle{plain}
 \newtheorem{thm}{Theorem}[section]
 \newtheorem{theorem}{Theorem}[section]      
 \newtheorem{cor}[thm]{Corollary}
 \newtheorem{lemma}[thm]{Lemma}              
\theoremstyle{definition}
 \newtheorem{definition}[thm]{Definition}    
 \newtheorem{exmp}[thm]{Example}

\theoremstyle{remark}
 \newtheorem{remark}[thm]{Remark}            

\usepackage{graphicx, color}
\usepackage{tikz}
\numberwithin{equation}{section}

\DeclareMathOperator{\Ad}{Ad}
\DeclareMathOperator{\mc}{mc}
\DeclareMathOperator{\md}{md}
\DeclareMathOperator{\me}{me}
\DeclareMathOperator{\Sp}{Sp}

\def\p{\partial}
\if\ART
\title{Middle convolutions of KZ-type equations and single-elimination tournaments}
\author{}
\date{}
\fi
\begin{document}
\if\ART
\maketitle
\else
\title[KZ-type equations and tournaments]
{Middle convolutions of KZ-type equations and single-elimination tournaments}

\author{Toshio Oshima}
\email{oshima@ms.u-tokyo.ac.jp}

\fi

\if\ART

\vspace*{-1.3cm}
\centerline{By}

\medskip
\centerline{Toshio \sc Oshima}
\centerline{\small (Josai University, Japan)}

\bigskip
\begin{quote}
{\small%
{\bf Abstract.} \ 
We introduce an extension of the generalized Riemann scheme for Fuchsian ordinary differential equations in the case of KZ-type equations.
This extension describes the local structure of equations obtained by resolving the singularities of KZ-type equations.
We present the transformation of this extension under middle convolutions.
As a consequence, we derive the corresponding transformation of the eigenvalues and multiplicities of the residue matrices of KZ-type equations under middle convolutions.
We interpret the result in terms of the combinatorics of single-elimination tournaments.\\
\quad{\sl\phantom{W} Key Words and Phrases.} \ 
middle convolution, KZ-type equation,Pfaffian system\\
\quad{\sl\phantom{W} 2020 Mathematics Subject Classification Numbers.} \ 
58A17, 34M35
}
\end{quote}
\else
\begin{abstract}
We introduce an extension of the generalized Riemann scheme for Fuchsian ordinary differential equations in the case of KZ-type equations.
This extension describes the local structure of equations obtained by resolving the singularities of KZ-type equations.
We present the transformation of this extension under middle convolutions.
As a consequence, we derive the corresponding transformation of the eigenvalues and multiplicities of the residue matrices of KZ-type equations under middle convolutions.
We interpret the result in terms of the combinatorics of single-elimination tournaments.
\end{abstract}
\maketitle
\fi
\section{Introduction}
A Fuchsian system
\[
 \mathcal N : \frac{du}{dx}=\sum_{\nu=1}^{n-1}\frac{A_{0\nu}}{x-x_\nu}u
\]
for a column vector $u$ of $N$ unknown functions has singularities at 
$n$ points
\if\ART
$x_1,\dots,x_{n-1},\infty\in\mathbb C\cup\{\infty\}$.
\else
$x_1,\dots,$ $x_{n-1},\infty\in\mathbb C\cup\{\infty\}$.
\fi
The local structure of the solution to $\mathcal N$ in a complex 
neighborhood of the singular point 
$x_\nu$ is characterized by the conjugacy class of the residue matrix 
$A_{0\nu}\in M(N,\mathbb C)$, namely,
the eigenvalues and their multiplicities of the matrix. 
Here the residue matrix at $\infty$ equals $A_{0\infty}=-(A_{01}+\cdots+A_{0\infty})$.
The data of the eigenvalues and their multiplicities of the residue matrices
are called (generalized) Riemann scheme of $\mathcal N$. 
The equation $\mathcal N$ is called rigid if it is irreducible and 
determined by the Riemann scheme.
If it is not determined by the Riemann scheme, $\mathcal N$ has finite number of 
accessory parameters.
For example, there are 188 rigid Fuchsian systems with order 
at most 8 and they are listed in \cite[\S13.2.3]{Ow}.

Katz \cite{katz1996rigid} shows that any 
rigid Fuchsian system $\mathcal N$ can be transformed into 
the trivial equation $u'=0$ by a successive application of 
invertible transformations called additions and middle convolutions. 
Haraoka \cite{Ha} shows that if $\mathcal N$ 
is rigid, it is extended to Knizhnik-Zamolodchikov-type (KZ-type) equation 
\[
  \mathcal M : 
  \frac{\p u}{\p x_i}=\sum_{\substack{0\le\nu<n\\ \nu\ne i}}  \frac{A_{i\nu}}{x_i-x_\nu}u 
  \quad(i=0,\dots,n{-}1)
\]
by putting $x_0=x$ and regarding the singular points 
$x_1,\dots,x_{n-1}$ as new variables.
This is proved by extending the middle convolution to KZ-type equations.
The middle convolution of $\mathcal N$ is explicitly 
given by Dettweiler-Reiter \cite{DR} in terms of $A_{0\nu}$ according to the definition by 
\cite{katz1996rigid} and Haraoka's result is an extension of this. 
These transformations with permutations of the variables $x_0,\dots,x_{n-1}$ 
or middle convolutions with respect to other variables $x_i$ keep the KZ-type but do 
not keep the rigidity in general. In fact, even if the original KZ-type equation have 
a rigid variable, the resulting KZ-type equation may have no rigid variable.

There are many known or new hypergeometric functions with several variables which
are solutions to KZ-type equations obtained by these transformations from the trivial equation
(cf.~\cite{MO2024, Oi}). 
It is useful to examine the middle convolution for the analysis of these equations. 
In the case of Fuchsian ordinary differential equations, we note that the most results in 
\cite{Ow} are obtained by such analysis. 
Since the middle convolution is a microlocal transformation, 
we can expect to have the corresponding
transformation of the fundamental properties of residue matrices describing the singularities, 
such as the eigenvalues and their multiplicities. 
In the case of the ordinary differential equation, they are obtained by \cite{DR}.
In the case of KZ-type equations, this problem is more difficult because the singularities are 
complicated. 
It is solved by \cite[Remark~3]{Okz} in the case of KZ-type equations with $n=4$, 
which correspond to hypergeometric functions of two variables. In this paper we study it
for general $n$.

The local solution of $\mathcal M$ at the singularity $x_i=x_j$ is characterized by 
the eigenvalues and their multiplicities of the residue matrix $A_{ij}$.
The local solution at the singularity $\{x_0=x_1,\ x_2=x_3\}$ is characterized by the conjugacy 
class of the pair $(A_{01},A_{23})$. Here $A_{01}$ commutes with $A_{23}$ because of the
integrability condition of $\mathcal M$ and therefore the simultaneous eigenvalues of
the matrices $A_{01}$ and $A_{12}$ and their multiplicities are the data which determine the
conjugacy class.
At the singularity $\{x_0=x_1=x_2\}$ the eigenvalues and their multiplicities of 
the matrix $A_{012}=A_{01}+A_{02}+A_{12}$ are important data characterizing the local solution.
In general, for $I\subset L_n:=\{0,1,\dots,n{-}1\}$ with $|I|>1$,
we define a generalized residue matrix $A_I$ as the sum of $A_{ij}$ satisfying 
$\{i,j\}\subset L_n$. 
A maximal commuting family  $\mathcal I$ of residue matrices is defined by 
a maximal subset of $\{A_I\mid I\subset L_n,\ |I|>1\}$ whose elements commute with each other.
Moreover we put $\Sp\mathcal M$ the data of simultaneous eigenvalues and their multiplicities 
of maximal commuting families of residue matrices of $\mathcal M$ (Definition~\ref{def:ComResFamily}) and  call $\Sp\mathcal M$ spectra of $\mathcal M$.
There are $(2n{-}3)!!$ maximal commuting families of residue matrices of $\mathcal M$ and 
there are $n{-}1$ commuting residue matrices in a maximal commuting family 
(Corollary~\ref{cor:numI}), which corresponds to the residue matrices with respect to 
hypersurfaces defining a normal crossing singularity in a blowing up of a singular point 
(cf.~\S\ref{sec:blowup}).

The main result of this paper is to show that the additions and middle convolutions of $\mathcal M$
induce transformations of $\Sp\mathcal M$ and give an explicit description of them (Theorem \ref{thm:main}).
For example, we need $\Sp\mathcal M$ to get the eigenvalues and their multiplicities of 
the residue matrix $A_{12}$ of the equation obtained by applying additions and middle convolutions 
several times to an original KZ-type equation $\mathcal M$ (Remark~\ref{rem:spectra}). 

Blowing up the singularities of the equation satisfied by Appell's hypergeometric series $F_1$, 
there appear 15 normal crossing singularities. 
The simplest KZ-type equation $\mathcal M$ is the equation of rank 3 with $n=4$ satisfied by
$F_1$ and in this case $\Sp\mathcal M$ consists of $(2n-3)!!=15$ decompositions
into simultaneous eigenspaces of residue matrices,  and 
the simultaneous eigenvalues are free in multiplicities 
and coincide with the characteristic exponents at 15 singular points.
$\Sp\mathcal M$ is an extension of the Riemann scheme of $\mathcal N$ to 
the KZ-type equation $\mathcal M$ (cf.~Example~\ref{ex:ODE}). 

A KZ-type equation $\mathcal M$ with $n=3$ naturally corresponds to
a Fuchsian system $\mathcal N$ with three singular points 
(cf.~\S\ref{sec:accessory}).
Moreover, many hypergeometric functions with $n{-}2$ variables are solutions to KZ-type
equations $\mathcal M$ (\cite{MO2024},\ \cite{Okz}). 
When $n=4$, a residue matrix $A_I$ with $|I|=3$ is expressed by another
residue matrix $A_J$ with $|J|=2$ (cf.~Example~\ref{exmpl:45} (i)) and $\Sp\mathcal M$
can be interpreted to the simultaneous eigenspace decompositions of pairs of 
residue matrices at normal crossing singular points and  
the transformation of $\Sp\mathcal M$ under the middle convolution is described in 
\cite{Okz}. 
When $n>4$, the structure of $\Sp\mathcal M$ is more complicated but 
the combinatorics of single-elimination tournaments helps to understand 
the structure.

In \S\ref{sec:tournament} some known results of 
combinatorics related to single-elimination tournaments, such as the number of ways in 
tournament scenarios, are explained with a focus on the application to the structure
of residue matrices of KZ-type equations. In particular, we introduce maximal families 
of commuting subsets of a finite set which correspond maximal commuting residue matrices.
The numbers parametrized by $n$ in this section can be referred to the data base 
\cite{OEIS}.

In \S\ref{sec:KZ}, after reviewing the integrability condition of a KZ-type equation, 
we introduce $\Sp\mathcal M$.

In \S\ref{sec:blowup}, giving a local coordinate system, 
we define a resolution of the singular point of a KZ-equation and show that
the maximal commuting family of residue matrices equals the set of 
residue matrices corresponding to the normal crossing divisors which define
a singular point in the blowing up.

In \S\ref{sec:midKZ}, we examine the transformation of residue matrices 
under the middle convolution of a KZ-type equation and give our main result 
Theorem~\ref{thm:main} in this paper.  To state the theorem, 
we define transformations of maximal families of commuting subsets of a finite set which
correspond transformations of tournaments.

Theorem~\ref{thm:main} can be applied to analyze hypergeometric functions with several variables.
Examples of the application with $n=4$ are given in \cite[\S8]{Okz}, \cite[\S7]{Oi}, 
\cite[\S5]{MO2024}.
A solution to a rigid Fuchsian system $\mathcal N$ with more than 3 
singular points can be deeply analyzed through the KZ-equation obtained from $\mathcal N$
by the extension of variables.
For example, the ordinary differential equation satisfied by Jordan-Pochhammer's hypergeometric 
function is extended to a KZ-type equation satisfied by Appell's $F_1$ or Lauricella's $F_D$. 
Moreover the theorem can be applied to KZ-type equations with irregular singularities 
through their versal unfoldings (cf.~\cite{Okzv}). 

In \S\ref{sec:Example}, Theorem~\ref{thm:main} is explained by examples with $n=4$.

Under a fractional linear transformation, KZ-type equation $\mathcal M$ can be assumed 
to have no singularity at infinity and 
in \S\ref{sec:infinity}, we examine the middle convolution of the equation.
In \S\ref{sec:fixedpt}, we study KZ-type equation $\mathcal M$ with fixed singular points.
In this case, $\Sp\mathcal M$ corresponds to single-elimination tournaments with a certain
restriction.  The equation has only one variable, our result corresponds
to that in \cite{DR}.
Accessory parameters and rigidity of a KZ-type equation and $\Sp\mathcal M$ for a 
Fuchsian holonomic system $\mathcal M$ are explained in \S\ref{sec:accessory}.
\section{Single-elimination tournaments}
\label{sec:tournament}
We consider single-elimination tournaments of $n$ teams.
In this paper tournaments always mean single-elimination tournaments.
A tournament of 3 teams distinguished by the labels $0$, $1$ and $2$ are
expressed by
\raisebox{-2mm}{\begin{tikzpicture}
\draw
(0,0)--(0,0.2)--(0.4,0.2)--(0.4,0)
;
\draw
(0.2,0.2)--(0.2,0.4)--(0.8,0.4)--(0.8,0)
(0.5,0.4)--(0.5,0.6)
node at (0,-0.2){$0$}
node at (0.4,-0.2){$1$}
node at (0.8,-0.2){$2$}
;
\end{tikzpicture}}
.
This chart means that team 0 and team 1 play first and the winner of this game and 
team 2 play the final game.
If the teams are not distinguished, there are two patterns 
\begin{tikzpicture}\ 
\draw
(0,0)--(0,0.2)--(0.4,0.2)--(0.4,0);
\draw
(0.2,0.2)--(0.2,0.4)--(0.8,0.4)--(0.8,0)
(0.5,0.4)--(0.5,0.6)
;
\end{tikzpicture} 
\ and \
\begin{tikzpicture}
\draw
(0.4,0)--(0.4,0.2)--(0.8,0.2)--(0.8,0);
\draw
(0,0)--(0,0.4)--(0.6,0.4)--(0.6,0.2)
(0.3,0.4)--(0.3,0.6);
\end{tikzpicture}
.
In the case of the tournaments of $n$ teams, the patterns correspond to binary 
one-rooted trees with $n{-}1$ non-leaf nodes and $n$ leaves figured in a plane.   
The total number of them are given by the Catalan number $C_{n-1}$. In this case, 
there are $n{-}1$ games and each game determines a 
final winner of some teams by a sub-tournament.

A tournament of three teams is determined by the teams playing the first game and there are 3 cases of the tournaments, which are described by
\begin{equation}\label{eq:teams}
 \bigl\{\{0,1\},\{0,1,2\}\bigr\},\ \bigl\{\{1,2\},\{0,1,2\}\bigr\},\ \bigl\{\{0,2\},\{0,1,2\}
 \bigr\}
\end{equation}
These are the sets of games in the tournaments and each game is labeled by a set of teams such that the game determines the final winner of the teams.  
Since the final game corresponds to the set of all teams, we may express the tournaments by the games excluding the final games, which we call shortened expression. 
Namely, the set of tournaments of three teams is 
\begin{equation}
 \bigl\{\{0,1\}\bigr\},\ \bigl\{\{1,2\}\bigr\},\  \bigl\{\{0,2\}\bigr\}
\end{equation}
in shortened expression.

If a tournament is changed to another tournament by a permutation of teams, we think that
these tournaments are isomorphic and call this isomorphic class a type of the tournament. 
The number of isomorphic classes of tournaments of $n$ teams is the $n$-th Wedderburn-Etherington 
number and a tournament can be expressed by $n{-}1$ subsets or $n{-}2$ proper subsets of 
the set of teams.

We moreover consider types of tournaments indicating the final winner and we call them win types.
There is one type of the tournaments of 3 teams and there are two win types of them.
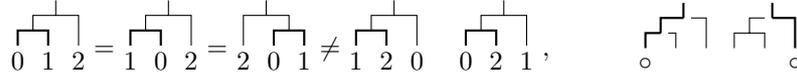
\[
\begin{tikzpicture}
\draw[thick]
(0,0)--(0,0.2)--(0.4,0.2)--(0.4,0)
;
\draw
(0.2,0.2)--(0.2,0.4)--(0.8,0.4)--(0.8,0)
(0.5,0.4)--(0.5,0.6)
node at (0,-0.2){$0$}
node at (0.4,-0.2){$1$}
node at (0.8,-0.2){$2$}
;
\end{tikzpicture}
\raisebox{3mm}{$=$}
\begin{tikzpicture}
\draw[thick]
(0,0)--(0,0.2)--(0.4,0.2)--(0.4,0)
;
\draw
(0.2,0.2)--(0.2,0.4)--(0.8,0.4)--(0.8,0)
(0.5,0.4)--(0.5,0.6)
node at (0,-0.2){$1$}
node at (0.4,-0.2){$0$}
node at (0.8,-0.2){$2$}
;
\end{tikzpicture}
\raisebox{3mm}{$=$}
\begin{tikzpicture}
\draw[thick]
(0.4,0)--(0.4,0.2)--(0.8,0.2)--(0.8,0)
;
\draw
(0,0)--(0,0.4)--(0.6,0.4)--(0.6,0.2)
(0.3,0.4)--(0.3,0.6)
node at (0,-0.2){$2$}
node at (0.4,-0.2){$0$}
node at (0.8,-0.2){$1$}
;
\end{tikzpicture}
\raisebox{3mm}{$\ne$}
\begin{tikzpicture}
\draw[thick]
(0,0)--(0,0.2)--(0.4,0.2)--(0.4,0)
;
\draw
(0.2,0.2)--(0.2,0.4)--(0.8,0.4)--(0.8,0)
(0.5,0.4)--(0.5,0.6)
node at (0,-0.2){$1$}
node at (0.4,-0.2){$2$}
node at (0.8,-0.2){$0$}
;
\end{tikzpicture}\ 
\ 
\begin{tikzpicture}
\draw[thick]
(0,0)--(0,0.2)--(0.4,0.2)--(0.4,0)
;
\draw
(0.2,0.2)--(0.2,0.4)--(0.8,0.4)--(0.8,0)
(0.5,0.4)--(0.5,0.6)
node at (0,-0.2){$0$}
node at (0.4,-0.2){$2$}
node at (0.8,-0.2){$1$}
;
\end{tikzpicture}
\raisebox{3mm}{,}
\qquad\quad
\begin{tikzpicture}
\draw[thick]
(0,0)--(0,0.2)--(0.2,0.2) 
(0.2,0.2)--(0.2,0.4)--(0.5,0.4) 
(0.5,0.4)--(0.5,0.6)
;
\draw
(0.3,0.2)--(0.4,0.2)--(0.4,0)
(0.6,0.4)--(0.8,0.4)--(0.8,0)
node at (0,-0.2) {$\circ$}
;
\end{tikzpicture}
\quad
\begin{tikzpicture}
\draw[thick]
(0.5,0.4)--(0.8,0.4)--(0.8,0)
(0.5,0.4)--(0.5,0.6)
;
\draw
(0,0)--(0,0.2)--(0.4,0.2)--(0.4,0)
(0.2,0.2)--(0.2,0.4)--(0.4,0.4) 
node at (0.8,-0.2){$\circ$}
;
\end{tikzpicture}\]
The above left hand side shows the expression of the tournament and the thick segment indicates
that the two tournaments selecting two players of the game are isomorphic by a suitable permutation
of teams. 
The right hand side gives 
the tournaments distinguished only by the final winner indicated by $\circ$.  

In the case of tournaments of three teams there are one type, three cases of tournaments, two patterns and two win types (1 type, 2 patterns, 3 tournaments, 2 win types).  

We give examples of the tournaments of 4 and 5 teams in shortened form: 

\medskip
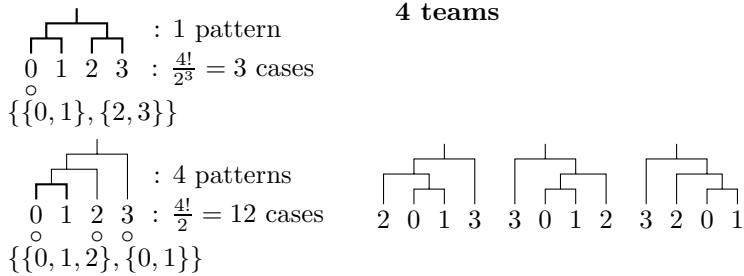
\centerline{\bf 4 teams}  

\vspace*{-2mm}
\begin{tikzpicture}
\draw[thick]
(0,0)--(0,0.2)--(0.4,0.2)--(0.4,0)
(0.8,0)--(0.8,0.2)--(1.2,0.2)--(1.2,0)
(0.2,0.2)--(0.2,0.4)--(1,0.4)--(1,0.2)
(0.6,0.4)--(0.6,0.6)
;
\draw
node at (0,-0.2){$0$}
node at (0.4,-0.2){$1$}
node at (0.8,-0.2){$2$}
node at (1.2,-0.2){$3$}
node at (2.45,0.3) {: 1 pattern}
node at (2.65,-0.2) {: $\frac{4!}{2^3}=3$ cases}
node at (0,-0.5){$\circ$}
node at (0.8,-0.8){$\{\{0,1\},\{2,3\}\}$}
;
\end{tikzpicture}

\begin{tikzpicture}
\draw[thick]
(0,0)--(0,0.2)--(0.4,0.2)--(0.4,0)
;
\draw
(0.2,0.2)--(0.2,0.4)--(0.8,0.4)--(0.8,0)
(0.4,0.4)--(0.4,0.6)--(1.2,0.6)--(1.2,0)
(0.8,0.6)--(0.8,0.8)
node at (0,-0.2){$0$}
node at (0.4,-0.2){$1$}
node at (0.8,-0.2){$2$}
node at (1.2,-0.2){$3$}
node at (2.45,0.3) {: 4 patterns}
node at (2.65,-0.2) {: $\tfrac{4!}2=12$ cases}
node at (0,-0.5){$\circ$}
node at (0.8,-0.5){$\circ$}
node at (1.2,-0.5){$\circ$}
node at (0.9,-0.8){$\{\{0,1,2\},\{0,1\}\}$}
;
\end{tikzpicture}
\quad
\begin{tikzpicture}
\draw
(0.4,0)--(0.4,0.2)--(0.8,0.2)--(0.8,0)
(0,0)--(0,0.4)--(0.6,0.4)--(0.6,0.2)
(0.4,0.4)--(0.4,0.6)--(1.2,0.6)--(1.2,0)
(0.8,0.6)--(0.8,0.8)
node at (0,-0.2){$2$}
node at (0.4,-0.2){$0$}
node at (0.8,-0.2){$1$}
node at (1.2,-0.2){$3$}
node at (1.2,-0.8){\phantom{0}}
;
\end{tikzpicture}
\begin{tikzpicture}
\draw
(0.4,0)--(0.4,0.2)--(0.8,0.2)--(0.8,0)
(0.6,0.2)--(0.6,0.4)--(1.2,0.4)--(1.2,0)
(0,0)--(0,0.6)--(0.8,0.6)--(0.8,0.4)
(0.4,0.6)--(0.4,0.8)
node at (0,-0.2){$3$}
node at (0.4,-0.2){$0$}
node at (0.8,-0.2){$1$}
node at (1.2,-0.2){$2$}
node at (1.2,-0.8){\phantom{0}}
;
\end{tikzpicture}
\begin{tikzpicture}
\draw
(0.8,0)--(0.8,0.2)--(1.2,0.2)--(1.2,0)
(0.4,0)--(0.4,0.4)--(1,0.4)--(1,0.2)
(0,0)--(0,0.6)--(0.8,0.6)--(0.8,0.4)
(0.4,0.6)--(0.4,0.8)
node at (0,-0.2){$3$}
node at (0.4,-0.2){$2$}
node at (0.8,-0.2){$0$}
node at (1.2,-0.2){$1$}
node at (1.2,-0.8){\phantom{0}}
;
\end{tikzpicture}\\[-1mm]
2 types, 5 (=1+4) patterns, 15 (=3+12) tournaments, 4 (=1+3) win types

\medskip

\noindent
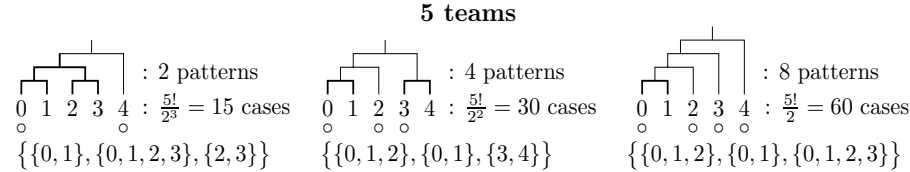
\scalebox{0.84}[0.9]{\begin{tikzpicture}
\draw[thick]
(0,0)--(0,0.2)--(0.4,0.2)--(0.4,0)
(0.8,0)--(0.8,0.2)--(1.2,0.2)--(1.2,0)
(0.2,0.2)--(0.2,0.4)--(1,0.4)--(1,0.2)
(0.6,0.4)--(0.6,0.6)
;
\draw
(0.6,0.4)--(0.6,0.6)--(1.6,0.6)--(1.6,0)
(1.1,0.6)--(1.1,0.8)
node at (0,-0.2){$0$}
node at (0.4,-0.2){$1$}
node at (0.8,-0.2){$2$}
node at (1.2,-0.2){$3$}
node at (1.6,-0.2){$4$}
node at (2.8,0.3) {: 2 patterns}
node at (3.05,-0.2) {: $\frac{5!}{2^3}=15$ cases}
node at (0,-0.5){$\circ$}
node at (1.6,-0.5){$\circ$}
node at (1.9,-0.9){$\bigl\{\{0,1\},\{0,1,2,3\},\{2,3\}\bigr\}$}
;
\end{tikzpicture}
\ 
\begin{tikzpicture}
\draw[thick]
(0,0)--(0,0.2)--(0.4,0.2)--(0.4,0)
(1.2,0)--(1.2,0.2)--(1.6,0.2)--(1.6,0)
;
\draw
(0.2,0.2)--(0.2,0.4)--(0.8,0.4)--(0.8,0)
(0.4,0.4)--(0.4,0.6)--(1.4,0.6)--(1.4,0.2)
(0.9,0.6)--(0.9,0.8)
node at (0,-0.2){$0$}
node at (0.4,-0.2){$1$}
node at (0.8,-0.2){$2$}
node at (1.2,-0.2){$3$}
node at (1.6,-0.2){$4$}
node at (2.8,0.3) {: 4 patterns}
node at (3.05,-0.2) {: $\frac{5!}{2^2}=30$ cases}
node at (0,-0.5){$\circ$}
node at (0.8,-0.5){$\circ$}
node at (1.2,-0.5){$\circ$}
node at (1.7,-0.9){$\bigl\{\{0,1,2\},\{0,1\},\{3,4\}\bigr\}$}
node at (2.2,1.2) {\large\bf 5 teams};
\end{tikzpicture}
\ 
\begin{tikzpicture}
\draw[thick]
(0,0)--(0,0.2)--(0.4,0.2)--(0.4,0)
;
\draw
(0.2,0.2)--(0.2,0.4)--(0.8,0.4)--(0.8,0)
(0.4,0.4)--(0.4,0.6)--(1.2,0.6)--(1.2,0)
(0.6,0.6)--(0.6,0.8)--(1.6,0.8)--(1.6,0)
(1.1,0.8)--(1.1,1.0)
node at (0,-0.2){$0$}
node at (0.4,-0.2){$1$}
node at (0.8,-0.2){$2$}
node at (1.2,-0.2){$3$}
node at (1.6,-0.2){$4$}
node at (0,-0.5){$\circ$}
node at (0.8,-0.5){$\circ$}
node at (1.2,-0.5){$\circ$}
node at (1.6,-0.5){$\circ$}
node at (2.8,0.3) {: 8 patterns}
node at (3.05,-0.2) {: $\frac{5!}{2}=60$ cases}
node at (1.9,-0.9){$\bigl\{\{0,1,2\},\{0,1\},\{0,1,2,3\}\bigr\}$}
;
\end{tikzpicture}}\\[-1mm]
3 types, 14 patterns, 105 (=15+30+60) tournaments, 9 (=2+3+4) win types

\medskip
Moreover we have 

\medskip
\noindent
\if\ART
\scalebox{0.89}[0.94]{
\begin{tabular}{|c|r|r|r|r|r|r|r|r|c|}
\multicolumn{10}{c}{\bf Numbers of single-elimination tournaments}
\\\hline
teams&2&3&4&5&6&7&8&9&$n$\\\hline\hline
patterns&1&2&5&14&42&132&429&1430&$T_n\!=\!\tfrac{(2n-2)!}{n!(n-1)!}$
\\\hline
win types&1&2&4&9&20&46&106&248&$W_n$\\ \hline
types&1&1&2&3&6&11&23&46&$U_n$
\\\hline
tournaments&1&3&15&105&945&10395&135135&2027025&$K_n\!=\!(2n-3)!!$\\ \hline
\end{tabular}}
\else
\ \ 
\scalebox{0.95}{
\begin{tabular}{|c|r|r|r|r|r|r|r|r|c|}
\multicolumn{10}{c}{\bf Numbers of single-elimination tournaments}
\\\hline
teams&2&3&4&5&6&7&8&9&$n$\\\hline\hline
patterns&1&2&5&14&42&132&429&1430&$T_n=\tfrac{(2n-2)!}{n!(n-1)!}$
\\\hline
win types&1&2&4&9&20&46&106&248&$W_n$\\ \hline
types&1&1&2&3&6&11&23&46&$U_n$
\\\hline
tournaments&1&3&15&105&945&10395&135135&2027025&$K_n=(2n-3)!!$\\ \hline
\end{tabular}}
\fi

\medskip
Here we denote the numbers of tournaments, types, win types, patterns by 
$K_n$, $U_n$, $W_n$ and $T_n$, respectively, for the tournaments of $n$ teams.  
Then we have the following recurrence relations.
\begin{align*}
T_n&=\sum_{k=1}^{n-1}T_k\cdot T_{n-k},\ \ T_1=1,
  \tag{patterns}\allowdisplaybreaks\\
W_n&=\sum_{k=1}^{n-1}W_k\cdot U_{n-k},\ \ W_1=1,\tag{win types}\allowdisplaybreaks\\
U_n&=\frac12\Bigl(\sum_{k=1}^{n-1}U_k\cdot U_{n-k}\ 
\bigl(+\,U_{\frac n2} \text{ if $n$ is even}\bigr)\ \Bigr),\ \ U_1=1,\tag{types}
\allowdisplaybreaks\\
K_n&=\frac12\sum_{k=1}^{n-1}{}_nC_k\cdot K_k\cdot K_{n-k},\ \ K_1=1. \tag{tournaments}
\end{align*}

Consider a tournament of $n$ teams.
The players of the final game are the winner of $k$ teams and that of the other $n{-}k$ teams with $k=1,\dots,n{-}1$.
They are determined by the tournament of $k$ teams and that of $n-k$ teams.
Considering the chart of the tournament under a suitable permutation of teams, we 
have the corresponding numbers of cases as follows.

\medskip
\hspace{8.2cm}\begin{tikzpicture}
\draw[densely dotted]
(0.2,-0.2)--(0.2,0)--(1,0)--(1,-0.2)
(0.6,0)--(0.6,0.15)
(1.8,-0.2)--(1.8,0)--(2.6,0)--(2.6,-0.2)
(2.2,0)--(2.2,0.15);
\draw
(0.6,0.45)--(0.6,0.6)--(1.4,0.6) (1.55,0.6)--(2.2,0.6)--(2.2,0.45)
(1.4,0.6)--(1.4,0.8)
node at (1.4,0.9) {\scriptsize$n$ teams}
node at (0.6,0.3) {\scriptsize$k$ teams}
node at (2.2,0.3) {\scriptsize$(n{-}k)$ teams}
node at (0.6,-0.5) {\small$0,\ldots,k{-}1$} 
node at (2.2,-0.5) {\small$k,\ldots,n$} 
;
\end{tikzpicture}

\vspace{-2cm}
\noindent
\begin{align*}
T_n&\leftarrow T_k\cdot T_{n-k}\qquad\qquad\ \,(1\le k< n),\\
W_n&\leftarrow W_k\cdot U_{n-k}\qquad\qquad(1\le k< n),\\
U_n&\leftarrow 
 \begin{cases}
 U_k\cdot U_{n-k}&(1\le k<n-k),\\
 \tfrac12 U_k(U_k-1)+U_k\phantom{AAAA}&(k=n-k),\\
 \end{cases}\\
K_n&\leftarrow \begin{cases}
   {}_nC_k\cdot K_k\cdot K_{n-k}&(1\le k<n-k),\\
   {}_nC_k\bigl(\tfrac12 K_k(K_k-1)\bigr)+
  \bigl(\tfrac12{}_nC_k\bigr)K_k&(k=n-k),\phantom{AAAAAAAAAAAAAAAAA}
   \end{cases}
\end{align*}
Hence, the recurrence relations of $T_n$ and $W_n$ are easily obtained because 
the first $k$ teams and the last $n-k$ teams are distinguished. 
To get the recurrence relations of $U_n$ and $K_n$, we may assume $k\le n-k$. 
Then taking care of 
the symmetry of the first $k$ teams and the last $n-k$ teams when $k=n-k$, we have 
their recurrence relations.

Consider a tournament of $n+1$ teams labeled by $0,1,\dots,n$. 
If the team $n$ does not participate the tournament, we delete the first game that the team 
is expected to play and we naturally have a tournament of $n$ teams as the following charts.
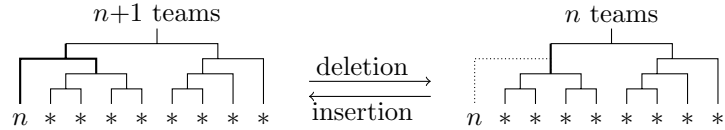
\[
\begin{tikzpicture}
\draw[thick]
(0.6,0.4)--(0.6,0.6)--(-0.4,0.6)--(-0.4,0)
(0.2,0.6)--(0.2,0.8);
\draw
(0,0)--(0,0.2)--(0.4,0.2)--(0.4,0)
(0.8,0)--(0.8,0.2)--(1.2,0.2)--(1.2,0)
(0.2,0.2)--(0.2,0.4)--(1,0.4)--(1,0.2)
(1.6,0)--(1.6,0.2)--(2,0.2)--(2,0)
(1.8,0.2)--(1.8,0.4)--(2.4,0.4)--(2.4,0)
(2,0.4)--(2,0.6)--(2.8,0.6)--(2.8,0)
(0.2,0.8)--(2.2,0.8)--(2.2,0.6)
(1.4,0.8)--(1.4,1.0)
node at (0,-0.2){$*$}
node at (0.4,-0.2){$*$}
node at (0.8,-0.2){$*$}
node at (1.2,-0.2){$*$}
node at (1.6,-0.2){$*$}
node at (2,-0.2){$*$}
node at (2.4,-0.2){$*$}
node at (2.8,-0.2){$*$}
node at (-0.4,-0.2) {$n$}
node at (1.4,1.2) {$n{+}1$ teams}
node at (4.1,0.5) {deletion}
node at (4.1,-0.1) {insertion}
;
\draw[->] (3.4,0.3)--(5,0.3);
\draw[<-] (3.4,0.1)--(5,0.1);
\end{tikzpicture}
\quad
\begin{tikzpicture}
\draw[thick] (0.6,0.4)--(0.6,0.8)
;
\draw[densely dotted]
(0.6,0.4)--(0.6,0.6)--(-0.4,0.6)--(-0.4,0)
;
\draw
(0,0)--(0,0.2)--(0.4,0.2)--(0.4,0)
(0.8,0)--(0.8,0.2)--(1.2,0.2)--(1.2,0)
(0.2,0.2)--(0.2,0.4)--(1,0.4)--(1,0.2)
(1.6,0)--(1.6,0.2)--(2,0.2)--(2,0)
(1.8,0.2)--(1.8,0.4)--(2.4,0.4)--(2.4,0)
(2,0.4)--(2,0.6)--(2.8,0.6)--(2.8,0)
(0.6,0.8)--(2.2,0.8)--(2.2,0.6)
(1.4,0.8)--(1.4,1.0)
node at (0,-0.2){$*$}
node at (0.4,-0.2){$*$}
node at (0.8,-0.2){$*$}
node at (1.2,-0.2){$*$}
node at (1.6,-0.2){$*$}
node at (2,-0.2){$*$}
node at (2.4,-0.2){$*$}
node at (2.8,-0.2){$*$}
node at (-0.4,-0.2) {$n$}
node at (1.4,1.2) {$n$ teams}
;
\end{tikzpicture}
\]
On the other hand, for a tournament of $n$ teams labeled by $0,1,\dots,n{-}1$, we can insert the 
first game of the team labeled by $n$ in any one of $2n{-}1$ vertical line segments in the chart. 
Namely, the new game is the first game of one of $n$ teams or the opponent of the game is a winner 
of one of the games of the tournament of $n$ teams.

The pair of a chart of a tournament of $n$
teams and one of its  vertical line segments and the chart of a tournament of $n{+}1$ teams
corresponds to each other by this operations, namely, deletion and insertion, respectively.
Hence we have
\begin{equation}
   K_{n+1}=(2n-1)K_n
.
\end{equation}
Moreover we have the following equalities by recurrence relations.
\begin{align*}
 K_n &= (2n-3)!!
     =\frac12 \sum_{k=1}^{n-1}{}_nC_k\cdot (2k-1)!!\cdot (2n-2k-1)!!
      ,\allowdisplaybreaks\\
 T_n &=\frac{(2n-2)!}{(n-1)!n!} 
     =\sum_{k=1}^{n-k}\frac{(2k-2)!}{(k-1)!k!}\frac{(2(n-k)-2)!}{(n-k-1)!(n-k)!},
  \allowdisplaybreaks\\
    1&=\Bigl(1-\sum_{k=1}^\infty U_nx^k\Bigr)\Bigl(\sum_{k=0}^\infty W_{k+1}x^k\Bigr).
\end{align*}
\begin{remark}\label{rem:basic-add} {\rm (i)} \ 
We say a deletion of a team is {\sl basic} when the first game of the team is also 
the first game of the opponent. An insertion of a team to a tournament is called {\sl basic} if  
the inserted game is the first game for the both players.
Note that the inversion of a basic deletion is a basic insertion and moreover at least 
one basic deletion is possible for any tournament.
Hence every tournament of $n$ teams is constructed by a successive application of basic 
insertions to a tournament of two teams.
Consequently, for given two tournaments of $n$ teams, a successive application 
of suitable $n{-}1$  pairs of a deletion and a basic insertion transforms one of the tournaments to the other.  Here we have only to look at the inserted games in this procedure.

{\rm (ii)} \ 
We consider the insertion which inserts a game played by the new team
and the winner of the original tournament,  and call it a top insertion. 
Applying successive top insertions to the game of two teams labeled by $0$ and $1$
we have a tournament of $n$ teams labeled by $0,\dots,n{-}1$ whose $n$-th game is 
played by the team with label $n{-}1$ and the selected team by the former games.

{\rm (iii)} \,The 4 diagrams given just before Remark~\ref{rem:spectra} show examples
of transformations of a tournament by a pair of a deletion and an insertion.
The last two diagrams correspond to a basic insertion and a top insertion, respectively.
\end{remark}

\bigskip
Let $L$ be a finite set with $|L|>1$.
Here $|L|$ denotes the cardinality of $L$.
\begin{definition}\label{def:ComSets}
$\mathcal I=\{I_\nu\mid \nu=1,\dots.r\}$ is \textit{a commuting family} of $L$ if
$|I_\nu|>1$ and $I_\nu\subset L$ for $\nu=1,\dots,r$ and 
\[
  I_\nu\cap I_{\nu'}=\emptyset\text{ or } I_\nu\subsetneqq I_{\nu'}\text{ or } I_\nu\supsetneqq I_{\nu'}
  \text{ \ for \ } 1\le \nu<\nu'\le r.
\]
Moreover if there is no commuting family $\mathcal I'$ of $L$ satisfying 
$\mathcal I \subsetneqq \mathcal I'$, 
$\mathcal I$ is called a {\sl maximal commuting family} of $L$. 
\end{definition}

Let $\mathcal I$ be a maximal commuting family of $L$ with $|L|>2$. 
If there is an element $I_0$ of $\mathcal I$ with $|I_0|=|L|-1$, then $\mathcal I\setminus \{L\}$
is a maximal commuting family of $I_0$.
If such $L_0$ does not exist in $\mathcal I$, there exist $I_1$ and $I_2$ in $\mathcal I$
which satisfy $I_1\cap I_2=\emptyset$ and
\[
  \mathcal I\setminus \{L\} =
  \mathcal I_1\cup\mathcal I_2,\ 
  \mathcal I_i:=\{I\in\mathcal I\mid  I \subset I_i\}\quad(i=1,\,2).
\]
Then $\mathcal I_i$ are maximal families of $I_i$ for $i=1$ and $2$.
Thus we have the following theorem.
\begin{theorem}
There is a  natural bijection of 
the set of maximal commuting families of $L$ onto the set of single-elimination tournaments of the teams labelled by $L$.  
\end{theorem}
Through the expression of the games of a tournament by 
the labels of the related teams as in \eqref{eq:teams},
a maximal commuting family $\mathcal I$ of $L$ corresponds to a single-elimination tournament of 
teams labeled by the elements of $L$.
In particular
\begin{align}
  |\mathcal I|=|L|-1.
\end{align}
\begin{definition}\label{def:bry}
For a maximal commuting family $\mathcal I$ of $L$, we put
\begin{equation}
 \widetilde{\mathcal I}:=\mathcal I\cup \bigcup_{\nu\in L}\bigl\{\{\nu\}\bigr\}
\end{equation}
and define 
\begin{equation}\label{eq:bI}
  b(I)=\{\bar b(I),\bar b'(I)\} \qquad(I\in\mathcal I,\ \bar b(I),\,\bar b'(I)
  \in\widetilde{\mathcal I})
\end{equation}
so that $I=\bar b(I)\sqcup\bar b'(I)$, namely,
\[ 
  \bar b(I)\cap \bar b'(I)=\emptyset,\ \ \bar b(I)\cup \bar b'(I)=I.
\]
Suppose the tournament corresponding to $\mathcal I$ is finished.
We consider that $\bar b$ specifies the losing side for each 
game in the tournament matches.
Moreover, suppose $i\in L$ is the final winner.
Then $i\not\in\bar b(I)$ for $I\in\mathcal I$. 
The map $\bar b$ satisfying this condition is denoted by $b^i$.
\end{definition}
Here, we give an example with $L=\{0,1,2,3,4\}$ and

\vspace{-6mm}
\begin{gather*}
  \quad\mathcal I=\bigl\{\{0,1\},\{0,1,2,3\},\{2,3\},\{0,1,2,3,4\}\bigr\}
\qquad\qquad\qquad\quad
\raisebox{-10mm}{\scalebox{0.9}{\begin{tikzpicture}
\draw
(-0.05,0)--(-0.05,0.2)--(0.05,0.2) (0.4,0.2)--(0.5,0.2)--(0.5,0)
(0.95,0)--(0.95,0.2)--(1.05,0.2) (1.45,0.2)--(1.55,0.2)--(1.55,0)
(0.25,0.3)--(0.25,0.6)--(0.35,0.6) (1.15,0.6)--(1.25,0.6)--(1.25,0.3)
;
\draw
(0.75,0.7)--(0.75,1)--(0.85,1) (1.9,1)--(2,1)--(2,0)
(1.4,1.1)--(1.4,1.3)
node at (0.25,0.2) {\small $01$}
node at (1.25,0.2) {\small $23$}
node at (0.75,0.6) {\small $0123$}
node at (1.375,1) {\small $01234$}
node at (0,-0.2){$0$}
node at (0.5,-0.2){$1$}
node at (1,-0.2){$2$}
node at (1.5,-0.2){$3$}
node at (2,-0.2){$4$}
;
\end{tikzpicture}}}\\[-10mm]
\intertext{and}
 \begin{split}
 b(\{0,1,2,3\})&=\bigl\{\{2,3\},\{0,1\}\bigr\},\qquad\,  b^0(\{0,1,2,3\})=\{2,3\},\\
 b(\{0,1,2,3,4\})&=\bigl\{\{4\},\{0,1,2,3\}\bigr\},\ b^0(\{0,1,2,3,4\})=\{4\},\\
 b(\{2,3\})&=\bigl\{\{2\},\{3\}\bigr\},\qquad\qquad\qquad b^0(\{2,3\})=\{2\}.
 \phantom{AAAAAAAAAAA}
\end{split}
\end{gather*}

Indicating the losing side by a gap of segment in a tournament chart, we obtain the following
left chart. Then the players of each game is determined by $\bar b$ as is indicated in 
the following right chart.  We denote the labels of the 
players of the game corresponding to $I\in\mathcal I$ by $I_{\bar b}$.  
Then 
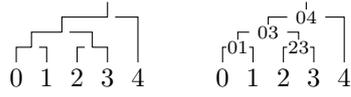
\begin{equation}\label{eq:players}
  I_{\bar b}:=I\setminus \bigcup_{I\supsetneqq J\in\mathcal I}\bar b(J),
\end{equation}
\qquad 
\raisebox{-6mm}{
\begin{tikzpicture}
\draw
(0,0)--(0,0.2)--(0.2,0.2) (0.3,0.2)--(0.4,0.2)--(0.4,0)
(0.8,0)--(0.8,0.2)--(0.9,0.2) (1,0.2)--(1.2,0.2)--(1.2,0)
(0.2,0.2)--(0.2,0.4)--(0.6,0.4) (0.7,0.4)--(1,0.4)--(1,0.2)
(0.6,0.4)--(0.6,0.6)--(1.2,0.6) (1.3,0.6)--(1.6,0.6)--(1.6,0)
(1.2,0.6)--(1.2,0.8)
node at (0,-0.2){$0$}
node at (0.4,-0.2){$1$}
node at (0.8,-0.2){$2$}
node at (1.2,-0.2){$3$}
node at (1.6,-0.2){$4$}
;
\end{tikzpicture}\qquad
\begin{tikzpicture}
\draw
(0,0)--(0,0.2)--(0.05,0.2) (0.35,0.2)--(0.4,0.2)--(0.4,0)
(0.8,0)--(0.8,0.2)--(0.85,0.2) (1.15,0.2)--(1.2,0.2)--(1.2,0)
(0.2,0.3)--(0.2,0.4)--(0.35,0.4) (0.85,0.4)--(1,0.4)--(1,0.3)
(0.6,0.5)--(0.6,0.6)--(0.85,0.6) (1.35,0.6)--(1.6,0.6)--(1.6,0)
(1.1,0.7)--(1.1,0.8)
node at (0,-0.2){$0$}
node at (0.4,-0.2){$1$}
node at (0.8,-0.2){$2$}
node at (1.2,-0.2){$3$}
node at (1.6,-0.2){$4$}
node at (0.2,0.2) {\scriptsize $01$}
node at (1.0,0.2) {\scriptsize $23$}
node at (0.6,0.4) {\scriptsize $03$}
node at (1.1,0.6) {\scriptsize $04$}
;
\end{tikzpicture}}
\qquad\qquad$\begin{cases}
\{0,1\}_{\bar b}=\{0,1\},\\
\{2,3\}_{\bar b}=\{2,3\},\\
\{0,1,2,3\}_{\bar b}=\{0,3\},\\
\{0,1,2,3,4\}_{\bar b}=\{0,4\}.
\end{cases}
$
\section{Spectra of KZ-type equations}
\label{sec:KZ}
A system of the equations
\begin{align}\label{eq:KZ}
\mathcal M : 
 \frac{\p u}{\p x_i}
  &=\sum_{\substack{0\le\nu\le n-1\\ \nu\ne i}}
    \frac{A_{i\nu}}{x_i-x_\nu}u
 \quad(i=0,\dots,n{-}1)\\[-6mm]\notag
\end{align}
with a vector 
$u=\left(\begin{smallmatrix}u_1\\[-1mm]\vdots\\ u_N\end{smallmatrix}\right)$ and 
matrices $A_{ij}=A_{ji}\in M(N,\mathbb C)$
is called a Knizhnik-Zamolodchikov-type (KZ-type) equation of rank $N$ (cf. [KZ]). 
Here $M(N,\mathbb C)$ denote the space of square matrices of size $N$ with elements in $\mathbb C$
and $A_{ij}$ is called the residue matrix along $x_i=x_j$ and 
they satisfy the integrability condition
\begin{align}
 [A_{ij},A_{k\ell}]&=0\qquad
  \bigl(\forall\{i,j,k,\ell\}\subset
 L_n\bigr),\\
 [A_{ij},A_{ik}+A_{jk}]&=0\qquad
  \bigl(\forall\{i,j,k\}\subset L_n\bigr),
\end{align}
where $i,j,k$ and $\ell$ are distinct indices.
Here and hereafter we use the notation
\begin{align}\label{eq:indices}
  L_n&:=\{0,1,\dots,n-1\},\quad \tilde L_n:=L_n\cup\{\infty\},\quad L_n^i:=L_n\setminus\{i\}.
\end{align}
\begin{definition}\label{def:ResMat}
General residue matrices are defined by 
\begin{align*}
A_{i\infty}&=-(A_{i0}+A_{i1}+\cdots+A_{i,n-1})\quad(i\in L_n),\\
A_{ii}&=A_\emptyset=A_i=0,\\
A_{i_1i_2\cdots i_k}&:=\displaystyle\sum_{1\le p<q\le k}A_{i_pi_q}\quad
\bigl(\{i_1,\dots,i_k\}\subset \tilde L_n\bigr).
\end{align*}
The matrix $A_{i\infty}$ is called the residue matrices of $\mathcal M$
along $x_i=\infty$.
\end{definition}
This definition and the integrability condition implies the following lemma.
\begin{lemma}[{\cite[\S2]{Okz}}] {\rm (i)} \ 
\qquad
$\displaystyle\sum_{i=0}^{n-1}A_{i\infty}=-2A_{L_n}$,

{\rm (ii)} \ 
For subsets $I$ and $J$ of $\tilde L_n$, we have
\begin{align}
  A_I-A_{\tilde L_n\backslash I}&=A_{L_n}
   &\bigl(I\subset L_n\bigr),\\
  [A_I,A_J]&=0&
  \bigl(I\cap J=\emptyset\text{ or }
  I\subset J\text{ or }I\supset J\bigr),\\
  [A_{L_n},A_I]&=0.
\end{align}
\end{lemma}

Since $[A_{L_n},A_{ij}]=0$ for $0\le i<j< n$, we have the following corollary. 
\begin{cor}\label{cor:kappa}
 Suppose the system $\mathcal M$ is irreducible, namely,
there is no nontrivial proper subspace $V\subset \mathbb C^N$ 
satisfying $A_{ij}V\subset V$ for $1\le i<j<n$.  Then
\begin{equation}
 A_{L_n}=\kappa I_N
\end{equation}
with a suitable $\kappa\in\mathbb C$.
\end{cor}
\begin{definition}
The transformation of $\mathcal M$ induced from the map $u\mapsto (x_p-x_q)^\lambda u$
is denoted by $\Ad\bigl((x_p-x_q\bigr)^\lambda)$, which transforms the residue matrix 
$A_{pq}$ to $A_{pq}+\lambda$ and does not change the other residue matrices $A_{i\nu}$
in \eqref{eq:KZ}.
\end{definition}
If $\mathcal M$ is irreducible, $\Ad\bigl((x_p-x_q\bigr)^{-\tau}\bigr)$ maps
$\mathcal M$ to the equation with $A_{L_n}=0$.
\begin{definition}\label{def:homog}
A KZ-type equation $\mathcal M$ is called {\bf homogeneous} if $A_{L_n}=0$.
In this case, we have
\begin{align}\label{eq:homg}
 \begin{split}
  A_I&=A_{\tilde L_n\setminus I}\quad\ \ (I\subset L_n),\\
 A_{i\infty}&=A_{\tilde L_n\setminus\{i\}}\quad (i\in L_n).
\end{split}
\end{align}
\end{definition}
\begin{remark}
Let $\mathcal I$ be a commuting family of subsets of $\tilde L_n$.
Then
\[
  \{I\in\mathcal I\mid\infty\not\in I\}\cup  \{\tilde L_n\setminus I\mid \infty\in I\in\mathcal I\}
\]
is a commuting family of subsets of $L_n$.
\end{remark}
\begin{remark}\label{rem:HG}
Since the residue matrices are invariant by the coordinate transformation 
$x_0\mapsto ax_0+b$, we may specialize $(x_1,x_2)=(0,1)$ without loss of generality.
By this specialization the number of variables is reduced to $n{-}2$.
Appell's $F_1$ and $F_2$ etc.\ are realized as solutions to certain KZ-type equations
with $n=4$ (cf.~\cite{MO2024}).
\end{remark}
\begin{definition}\label{def:ComResFamily}
The set $\bigl\{A_I\mid I\in\mathcal I\}$ defined by the 
maximal commuting family $\mathcal I$ of $L_n$
is called a {\bf maximal commuting family of residue matrices} 
of $\mathcal M$.
\end{definition}

The result in the former section implies the following.
\begin{cor}\label{cor:numI}
{\rm (i)} \ 
$|\mathcal I|=|L|-1$.

{\rm (ii)} \ 
$\mathcal M$ has $(2n-3)!!$ maximal commuting families of residue matrices.
\end{cor}

\begin{exmp}\label{exmpl:45}
We give examples of maximal commuting families of residue matrices.
Here $i,j,k,\ell,m,p,q$ are distinct indices.
Since the maximal family of residue matrices always contains $A_{L_n}$,
we give the family omitted $A_{L_n}$

{\rm (i)} \ 
When $n=4$, the families are 
\[\{A_{ij},A_{k\ell}\},\ \{A_{ij},A_{ijk}\}\qquad\bigl(\{i,j,k,\ell\}=\{0,1,2,3\}\bigr).\]
Moreover if $\mathcal M$ is homogeneous, we have $A_{0\infty}=A_{123}$ etc.\ and these 
$W_4=15$ families cf.~\cite{Okz}) are 
\[\{A_{ij},A_{k\ell}\}\qquad \bigl(\{i,j,k,\ell\}\subset\{0,1,2,3,\infty\}\bigr).\]

\smallskip
{\rm (ii)} \ 
When $n=5$, there are $W_5=105$ maximal commuting families of residue matrices (cf.~\cite[p.94]{Onarajp}) :
\begin{align*}
&\{A_{ij},A_{k\ell},A_{ijk\ell}\},\ \{A_{ij},A_{ijk},A_{\ell m}\},\ \{A_{ij},A_{ijk},A_{ijk\ell}\}
\\
&\qquad\qquad\bigl(\{i,j,k,\ell,m\}=\{0,1,2,3,4\}\bigr).
\end{align*}
Moreover if $\mathcal M$ is homogeneous, they are
\[
 \{A_{ij},A_{k\ell},A_{pq}\},\ \{A_{ij},A_{ijk},A_{pq}\}
 \qquad (\{i,j,k,\ell,p,q\}=\{0,1,2,3,4,5,\infty\}).
\]
\end{exmp}

\bigskip
\noindent
Now we review the notation introduced by \cite{Okz}.
The (generalized) eigenvalues and their multiplicities of a matrix $A\in M(N,\mathbb C)$
is written by
\[
  [A]=\{[\lambda_1]_{m_1},\dots,[\lambda_r]_{m_r}\}
\]
with $m_1+\cdots+m_r=N$ and $\prod_{i=1}^r(A-\lambda_i)=0$. 

Sometimes we assume $m_i\ge m_j$ if $\lambda_i=\lambda_j$ and moreover 
\[
 \mathrm{rank}\,
\prod_{\nu=1}^{k}(A-\lambda_\nu)=N-(m_1+\dots+m_k)\quad(k=1,\dots,r),
\]
or $A\in M(N,\mathbb C)$ is a limit of matrices satisfying this equality, 
but we do not assume them in this paper.

When two matrices $A,\,B\in M(N,\mathbb C)$ commute to each other, we can consider 
simultaneous eigenspace decompositions and we introduce the notation to the 
simultaneous eigenvalues and their multiplicities. For example, when
\[ 
A=\left(\begin{smallmatrix}0\\&0\\&&0\!\!\\&&&-1\end{smallmatrix}\right),\ 
B=\left(\begin{smallmatrix}1\\&2\\&&2\\&&&3\end{smallmatrix}\right),
\]
we have
\[
 [A]=\{[0]_3,[-1]_1\}{\,=\{[0]_3,-1\}},\  [B]=\{[1]_1,[2]_2,[3]_1\}{\,=\{1,[2]_2,3\}}
\]
and the simultaneous eigenvalues and their multiplicities are written by
\[[A:B]=\bigl\{[0:1]_1,[0:2]_2,[-1:3]_1\bigr\}{\,=\bigl\{[0:1],[0:1]_2,[-1:3]\bigr\}}.\]

In general, when $[B_i,B_j]=0\ \ (i,j=1,..,k)$,  we use the notation
\[
[B_1:\cdots:B_k]=
 \bigl\{[\lambda_{1,1}:\cdots:\lambda_{k,1}]_{m_1},\cdots,
 [\lambda_{1,r}:\cdots:\lambda_{k,r}]_{m_r}\bigr\}
\]
and for a positive integer $p$, we put
\[
[B_1:\cdots:B_k]_p=
 \bigl\{[\lambda_{1,1}:\cdots:\lambda_{k,1}]_{pm_1},\cdots,
 [\lambda_{1,r}:\cdots:\lambda_{k,r}]_{pm_r}\bigr\}.
\]

The matrices which commute to each other can be simultaneously transformed 
to upper triangular matrices under a suitable base.

\smallskip
\noindent
\begin{definition}\label{def:Sp}
Let $\mathcal L_n$ be the set of maximal commuting families of $L_n$.
The {\bf spectra} of $\mathcal M$ is defined by  
\begin{align*}
 \Sp\,\mathcal M&:=\bigl\{[A_{I_1}:\dots:A_{I_{n-1}}]\mid \mathcal I
  =\{I_1,\dots,I_{n-1}\}\bigr\}_{\mathcal I\in\mathcal L_n},\\
 \Sp' \mathcal M&:=\bigl\{[A_{I_1}:\dots:A_{I_{n-2}}]
  \mid \mathcal I
 =\{I_1,\dots,I_{n-2},L_n\}\bigr\}_{\mathcal I\in\mathcal L_n}
.
\end{align*}
In most cases, $A_{L_n}$ is a scalar matrix or 0 and we use $\Sp' \mathcal M$ in place of $\Sp\,\mathcal M$. For example, the transformation used in \cite[\S5]{MO2024} keeps the homogeneity and
therefore we may assume $A_{L_n}=0$.
\end{definition}

\begin{remark}\label{rem:spsum} 
Let $\{I_1,\dots,I_k\}$ be a maximal commuting family of $L_n$.
Then $[\sum_i c_iA_{I_i}]$ $(c_i\in\mathbb C)$ is obtained from $\Sp\mathcal M$.
For example, since $A_{01}+\cdots+A_{0k}=A_{01\cdots k}-A_{1\cdots k}$, 
we get $[A_{01}+\cdots+A_{0k}]$ from $\Sp\mathcal M$
(cf.~\cite{Osemilocal}, \S\ref{sec:semilocal}). 
\end{remark}

\section{Blowing up of singular points}
\label{sec:blowup}

Note that the KZ-type equation \eqref{eq:KZ}
is written in the Pfaffian form
\begin{equation}
du=\Omega u,\quad \Omega=\sum_{0\le i<j<n}\!\!\!
A_{ij}d\log(x_i-x_j). 
\end{equation}

Let $\mathcal I=\{I_1,\dots,I_{n-2},I_{n-1}=L_n\}$ be an element of $\mathcal L_n$.
Namely, $\mathcal I$ is a maximally commuting family of subsets of
$L_n=\{0,1,\ldots,n{-}1\}$. 
Using the notation given in Definition~\ref{def:bry} and \eqref{eq:players}, we put
\begin{align}
b(L_n)&=\{J,\,J'\},\ \ J=\{j_0,\dots,j_k\},\ \  J'=\{j'_0,\dots,j'_{k'}\},
\label{eq:semifinal}\\
(I_i)_{\bar b}&=\{n_i,n'_i\},\quad n_i,\,n'_i\in L_n.
\end{align}
Then $k+k'=n-2$ and $k\ge0$, $k'\ge 0$, 
and $J$ and $J'$ correspond
semi-final matches of the tournament $\mathcal I$.  
Note that $n_i$ and $n_i'$ correspond to the players of the match $I_i$, which 
are determined by the result of games of the tournament.
\noindent
Moreover $b(L_n)$ corresponds to the singular point 
\begin{equation}\label{eq:singpt}
 x_{j_0}=\cdots=x_{j_k} \text{ and } x_{j'_0}=\cdots=x_{j'_{k'}}
\end{equation}
of the KZ-type equation
and the local coordinate system $(x_{j_1},\dots,x_{j_k},x_{j'_1},\dots,x_{j'_{k'}})$ 
with putting $x_{j_0}=0$ and 
$x_{j'_0}=1$ is valid in a neighborhood of the point.
Note that $x_{j_\nu}-x_{j'_{\nu'}}$ does not vanish at the neighborhood.
We give a local coordinate system which gives a resolution of the singular point as follows.

\smallskip
\noindent
\begin{definition}
Define local coordinate system $(X_1,\dots,X_{n-2})$ by
\begin{equation}
x_{n_i}-x_{n'_i}=\displaystyle\prod_{I_i\subset I_j\ne L_n} X_j.
\end{equation}
\end{definition}
\begin{remark}
We may assume $n_i\in \bar b(I_i)$ and $n'_i\in\bar b'(I_i)$. Then the condition
\begin{equation}
 \begin{split}
 |x_{n_i}-x_\nu|\le \epsilon|x_{n_i}-x_{n'_i}|,\ 
   |x_{n'_i}-x_{\nu'}|\le \epsilon|x_{n_i}-x_{n'_i}|\\
 (\nu\in \bar b(I_i),\ \nu'\in\bar b'(I_i),\ i\in L_n)
 \end{split}
\end{equation}
with $0<\epsilon\ll1$ corresponds to a neighborhood of the origin of 
the local coordinate $(X_1,\dots,X_{n-2})$. 
Considering $\mathcal I\in\mathcal L$ satisfying \eqref{eq:semifinal}, the following theorem
gives a resolution of the singular point \eqref{eq:singpt}.
\end{remark}
\noindent
\begin{theorem} The 1-form 
$\Omega -\sum\limits_{i=1}^{n-2}A_{I_i}d\log(X_i)$ \ 
is non-singular around the origin of $(X_1,\dots,X_{n-2})$.
\end{theorem}

\begin{proof}
Under the notation in Lemma~\ref{lem:sumbase} for $1\le i<j<n$, 
the condition $I_{i,j}\subset I$ for $I\in\mathcal I$ is equals to $\{i,j\}\subset I$ and therefore the lemma shows
\begin{equation}
  x_i-x_j=f_{i,j}(X)\cdot\!\!\! \prod_{I_{i,j}\subset I_\nu\in\mathcal I\setminus\{L_n\}}X_\nu.
\end{equation}
Here $f_{i,j}$ is a polynomial of $X_\nu$ satisfying $I_\nu\subsetneqq I_{i,j}$.
Note that the constant term of $f_{i,j}$ equals  $1$ or $-1$.
Hence
\begin{equation}
  d\log(x_i-x_j)\,-\!\!\!\!\sum\limits_{\{i,j\}\subset I_\nu\in\mathcal I\setminus\{L_n\}}
  \!\!\!\!\!A_I d\log{X_\nu}
\end{equation}
has no singularity around the origin. 
\end{proof}
\begin{lemma}\label{lem:sumbase}
Let $\mathcal I$ be a maximal commuting family of $L=\{0,\dots,n-1\}$ and
$\bar b$ is a map of $\mathcal I$ given in Definition~\ref{def:bry}.  Put
\begin{equation}
  x_I:=x_{n_I}-x_{n'_I} \text{ \ with \ } I_{\bar b}=\{n_I,n_I'\}\subset L
\end{equation}
for $I\in\mathcal I$. Moreover, for $i,\,j\in L$ with $i\ne j$, we define 
$I_{i,j}$ the minimal subset in $\mathcal I$
containing both $i$ and $j$.
Then $x_I$ ($I\in\mathcal I$) are linearly independent over $\mathbb C$ and
\begin{equation}\label{eq:sumbase}
x_i-x_j=\sum_{I\in \mathcal I}\epsilon_{i,j}^I x_I\text{ \ with \ }
 \begin{cases}
   \epsilon_{i,j}^I= 0 &(I\supsetneqq I_{i,j}\text{ or }I\cap I_{i,j}=\emptyset),\\
   \epsilon_{i,j}^I \in\{1, -1\}&(I=I_{i,j}),\\
   \epsilon_{i,j}^I \in\mathbb Z&(I\subsetneqq I_{i,j}).
 \end{cases}
\end{equation}
\end{lemma}
\begin{proof}
Put $(I_{i,j})_{\bar b}=\{k,\ell\}$. 
Note that the lemma is clear when $|I_{i,j}|=2$.
We will prove \eqref{eq:sumbase} by the induction on the cardinality of $I_{i,j}$.
We may assume $i\in \bar b(I_{i,j})$ and $j\in\bar b'(I_{i,j})$ by swapping $i$ and $j$
if necessary. 
Similarly we may moreover assume $k\in \bar b(I_{i,j})$ and $\ell\in \bar b'(I_{i,j})$.
Then $i=k$ or $I_{i,k}\subsetneqq I_{i,j}$.
Moreover $j= \ell$ or  $I_{j,\ell}\subsetneqq I_{i,j}$.
Since $x_i-x_j=(x_k-x_\ell)+(x_i-x_k)-(x_j-x_\ell)$, the hypothesis of the induction
proves \eqref{eq:sumbase}.
Since the dimension of $\sum_{i,j\in L}\mathbb C(x_i-x_j)$ equals $n{-}1$ and $|\mathcal I|=n{-}1$,
$x_I$ ($I\in\mathcal I$) are linearly independent.
\end{proof}

\centerline{\bf Examples}
Put $\Omega':=\sum\limits_{i=1}^{n-2}A_{I_i} d\log X_i$ in the theorem.

\medskip
\centerline{$\bf n=4$}

\vspace{1.1cm}\hspace{9.8cm}
\scalebox{0.8}{\begin{tikzpicture}
\draw (-0.5,0)--(1.5,0)  (-0.5,1)--(1.5,1) (0,-0.5)--(0,1.5)  (1,-0.5)--(1,1.5) 
 (-0.5,-0.5)--(1.5,1.5)
node at (1.7,0) {$x$}
node at (0,1.7) {$y$}
node at (-0.3,0.2) {$0$}
node at (0.2,-0.3) {$0$}
node at (-0.3,1.2) {$1$}
node at (1.2,-0.3) {$1$}
;
\end{tikzpicture}}
\vspace{-3.1cm}

\smallskip
$\displaystyle
\Omega=A_{x0}\frac{dx}x+A_{y0}\frac{dy}y+A_{x1}\frac{d(x-1)}{x-1}
+A_{y1}\frac{d(y-1)}{y-1}+A_{xy}\frac{d(x-y)}{x-y}.
$

\vspace{-2mm}
\begin{align*}
(x,y)&=(0,1) : \quad 
\raisebox{-5mm}{
\begin{tikzpicture}
\draw
(0,0)--(0,0.2)--(0.2,0.2) (0.3,0.2)--(0.4,0.2)--(0.4,0)
(0.8,0)--(0.8,0.2)--(0.9,0.2) (1,0.2)--(1.2,0.2)--(1.2,0)
(0.2,0.2)--(0.2,0.4)--(0.6,0.4) (0.7,0.4)--(1,0.4)--(1,0.2)
(0.6,0.4)--(0.6,0.6)
node at (0,-0.2){$0$}
node at (0.4,-0.2){$x$}
node at (0.8,-0.2){$y$}
node at (1.2,-0.2){$1$}
node at (0,-0.5){$x_0$}
node at (0.4,-0.5){$x_1$}
node at (0.8,-0.5){$x_2$}
node at (1.2,-0.5){$x_3$}
;
\end{tikzpicture}}
&&\begin{cases}
x_1-x_0=x=X,\\
x_3-x_2=1-y=Y,
\end{cases}\\
&\phantom{(x,y)=(0,0)}
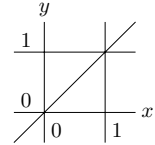
\begin{cases}
X=x,\\
Y=1-y,
\end{cases}
&&
\begin{cases}
 \frac{dx}x=\frac{dX}X,\\
 \frac{dy}y=\frac{dY}Y,
\end{cases}\\
&\Omega'=A_{x0}\frac{dX}X+A_{y1}\frac{dY}{Y}&& (|x|,\,|y-1|\ll1)
\allowdisplaybreaks\\[2mm]
(x,y)&=(0,0) : 
\quad
\raisebox{-6mm}
{\begin{tikzpicture}
\draw (0,0)--(0,0.2)--(0.2,0.2) (0.3,0.2)--(0.4,0.2)--(0.4,0)
(0.2,0.2)--(0.2,0.4)--(0.4,0.4) (0.5,0.4)--(0.8,0.4)--(0.8,0)
(0.4,0.4)--(0.4,0.6)--(0.6,0.6) (0.7,0.6)--(1.2,0.6)--(1.2,0)
(0.6,0.6)--(0.6,0.8)
node at (0,-0.2){$0$}
node at (0.4,-0.2){$x$}
node at (0.8,-0.2){$y$}
node at (1.2,-0.2){$1$}
node at (0,-0.5){$x_0$}
node at (0.4,-0.5){$x_1$}
node at (0.8,-0.5){$x_2$}
node at (1.2,-0.5){$x_3$}
;
\end{tikzpicture}}
&&\begin{cases}
  x_2-x_0=y=Y,\\
  x_1-x_0=x=XY,\\
  x_2-x_1=y-x=(1-X)Y,
\end{cases}\\
&\phantom{(x,y)=(0,1)}
\begin{cases}
 X=\frac xy,\\
 Y=y,
\end{cases}
&&
\begin{cases}
 \frac{dx}x=\frac{dX}X+\frac{dY}Y,\\
 \frac{d(x-y)}{x-y}=\tfrac{dY}{Y}+\tfrac{d(X-1)}{X-1},
\end{cases}\\
&\Omega'=A_{x0}\frac{dX}{X}+A_{xy0}\frac{dY}{Y},&&
A_{xy0}:=A_{x0}+A_{y0}+A_{xy}\quad(|x|\ll|y|\ll1).
\end{align*}
\begin{remark}
In \cite{MO2024,Oi}, the local coordinate $(X,Y)=(\frac yx,y)$ is used for 
a desingularization of the origin, 
where $(X,Y)$ is in a neighborhood of $(\infty,0)$. 
This coordinate transformation keeps KZ-type equations and the point $(\infty,0)$ is
a normal crossing singular point of the equations.
\end{remark}

\centerline{$\bf n=5$}
\noindent
\begin{align*}
\Omega&=A_{x0}\frac{dx}x+A_{y0}\frac{dy}y+A_{z0}\frac{dz}z+A_{x1}\frac{d(x-1)}{x-1}
  +A_{y1}\frac{d(y-1)}{y-1}+A_{z1}\frac{d(z-1)}{z-1}\\
 &\quad{}+A_{xy}\frac{d(x-y)}{x-y}
  +A_{yz}\frac{d(y-z)}{y-z}+A_{xz}\frac{d(x-z)}{x-z}.
\end{align*}

\begin{align*}
(x,y,z)&=(0,0,1):\ 
\raisebox{-0.7cm}{\begin{tikzpicture}
\draw
(0,0)--(0,0.2)--(0.2,0.2) (0.3,0.2)--(0.4,0.2)--(0.4,0)
(1.2,0)--(1.2,0.2)--(1.3,0.2) (1.4,0.2)--(1.6,0.2)--(1.6,0)
(0.2,0.2)--(0.2,0.4)--(0.4,0.4) (0.5,0.4)--(0.8,0.4)--(0.8,0)
(0.4,0.4)--(0.4,0.6)--(0.8,0.6) (0.9,0.6)--(1.4,0.6)--(1.4,0.2)
(0.8,0.6)--(0.8,0.8)
node at (0,-0.2){$0$}
node at (0.4,-0.2){$x$}
node at (0.8,-0.2){$y$}
node at (1.2,-0.2){$z$}
node at (1.6,-0.2){$1$}
node at (0,-0.5){$x_0$}
node at (0.4,-0.5){$x_1$}
node at (0.8,-0.5){$x_2$}
node at (1.2,-0.5){$x_3$}
node at (1.6,-0.5){$x_4$}
;
\end{tikzpicture}}
&&
\begin{cases}
 x_2-x_0=y=Y,\\
 x_1-x_0=x=XY,\\
 x_4-x_3=1-z=Z,
\end{cases}&&\\
&\quad\begin{cases}
 X=\frac xy\\
 Y=y,\\
 Z=1-z,
\end{cases}
&&\begin{cases}
 x-y=(X-1)Y,\\
 z=1-Z,
\end{cases}\\
&\Omega'=A_{x0}\frac{dX}X+A_{xy0}\frac{dY}{Y}+A_{z1}\frac{dZ}{Z}&&
(|x|\ll |y|\ll 1,\ |z-1|\ll1)\allowdisplaybreaks\\[2mm]
(x,y,z)&=(0,0,0):\ 
\raisebox{-9mm}{
\begin{tikzpicture}
\draw
(0,0)--(0,0.2)--(0.2,0.2) (0.3,0.2)--(0.4,0.2)--(0.4,0)
(0.2,0.2)--(0.2,0.4)--(0.4,0.4) (0.5,0.4)--(0.8,0.4)--(0.8,0)
(0.4,0.4)--(0.4,0.6)--(0.6,0.6) (0.7,0.6)--(1.2,0.6)--(1.2,0)
(0.6,0.6)--(0.6,0.8)--(0.8,0.8) (0.9,0.8)--(1.6,0.8)--(1.6,0)
(0.8,0.8)--(0.8,1)
node at (0,-0.2){$0$}
node at (0.4,-0.2){$x$}
node at (0.8,-0.2){$y$}
node at (1.2,-0.2){$z$}
node at (1.6,-0.2){$1$}
node at (0,-0.5){$x_0$}
node at (0.4,-0.5){$x_1$}
node at (0.8,-0.5){$x_2$}
node at (1.2,-0.5){$x_3$}
node at (1.6,-0.5){$x_4$}
;
\end{tikzpicture}}
&&
\begin{cases}
x_3-x_0=z=Z,\\
x_2-x_0=y=YZ,\\
x_1-x_0=x=XYZ,
\end{cases}\\
&
\qquad
\begin{cases}
X=\frac xy,\\
Y=\frac yz,\\
Z=z,
\end{cases}
&&
\begin{cases}
y-z=(Y-1)Z,\\
x-y=(X-1)YZ,\\
x-z=(XY-1)Z,\\
\end{cases}\\
&\Omega'=A_{x0}\frac{dX}X+A_{xy0}\frac{dY}{Y}+A_{xyz0}\frac{dZ}{Z}
&&(|x|\ll|y|\ll|z|\ll1)\allowdisplaybreaks\\[2mm]
(x,y,z)&=(0,0,0):\ 
\raisebox{-0.9cm}{\begin{tikzpicture}
\draw
(0,0)--(0,0.2)--(0.2,0.2) (0.3,0.2)--(0.4,0.2)--(0.4,0)
(0.8,0)--(0.8,0.2)--(0.9,0.2) (1,0.2)--(1.2,0.2)--(1.2,0)
(0.2,0.2)--(0.2,0.4)--(0.5,0.4) (0.6,0.4)--(1,0.4)--(1,0.2)
(0.6,0.4)--(0.6,0.6)--(0.7,0.6) (0.8,0.6)--(1.6,0.6)--(1.6,0)
(0.8,0.6)--(0.8,0.8)
node at (0,-0.2){$0$}
node at (0.4,-0.2){$x$}
node at (0.8,-0.2){$y$}
node at (1.2,-0.2){$z$}
node at (1.6,-0.2){$1$}
node at (0,-0.5){$x_0$}
node at (0.4,-0.5){$x_1$}
node at (0.8,-0.5){$x_2$}
node at (1.2,-0.5){$x_3$}
node at (1.6,-0.5){$x_4$};
\end{tikzpicture}}
&&
\begin{cases}
 x_3-x_0=z=Z,\\
 x_1-x_0=x=XZ,\\
 x_3-x_2=z-y=YZ,
\end{cases}\\
&\qquad
\begin{cases}
 X=\frac xz,\\
 Y=\frac{z-y}z,\\
 Z=z,
\end{cases}&&
\begin{cases}
 y=(1-Y)Z,\\
 x-y=(X+Y-1)Z,\\
 x-z=(X-1)Z,
\end{cases}\\
&\Omega'=A_{x0}\frac{dX}X+A_{yz}\frac{dY}{Y}+A_{xyz0}\frac{dZ}{Z}
&&(|x|,\,|y-z|\ll|z|\ll1).
\end{align*}

\section{Middle convolutions of KZ-type equations}
\label{sec:midKZ}
The {convolution} $\widetilde{\mathrm{mc}}_{x_0,\mu}\,\mathcal M$ of 
${\mathcal M}$ \eqref{eq:KZ} with ${\mu}\,\in\mathbb C$ is defined by
{\small\begin{align*}
 {\widetilde{\mathcal M}}\,:\,\frac{\p \tilde u}{\p x_i}=\sum_{0\le \nu< n}\!\!\!\frac{\tilde A_{i\nu}}{x_i-x_\nu}\tilde u\quad(0\le i<n),
\end{align*}}
\\[-.8cm]
{\small\begin{align*}
\tilde A_{0k}&=
 \bordermatrix{
 & & &k\cr
 & 0&\cdots&0&\cdots&0\cr
 & \vdots&\cdots&\vdots&\cdots&\vdots\cr
k & A_{01}&\cdots&A_{0k}\,{+\mu}&\cdots&A_{0n-1}\cr
 & \vdots&\cdots&\vdots&\cdots&\vdots\cr
 & 0&\cdots&0&\cdots&0\cr}
\begin{matrix}
\phantom{AAAA}\\[2mm]
\in M\bigl((n-1)N,\mathbb C\bigr)\phantom{AAAA}\\[1mm]
\qquad\text{\small(Dettweiler-Reiter \cite{DR})},
\end{matrix}
\allowdisplaybreaks\\
 \tilde A_{ij}&=
  \bordermatrix{
 & & &i&&j\cr
 & A_{ij}\cr
 &&\ddots\cr
 i&&&A_{ij}\,{+A_{0j}}&&-A_{0j}\\
 &&&&\ddots\cr
 j&&&-A_{0i}&&A_{ij}\,{+A_{0i}}\cr
 &&&&&&\ddots\cr
 &&&&&&&A_{ij}\cr}
\begin{matrix}
\phantom{AAA}\\[2mm]
\in M\bigl((n-1)N,\mathbb C\bigr)\\[1mm]
\quad\text{\quad\small(Haraoka \cite{Ha}).}
\end{matrix}\notag
\end{align*}}
In particular, the compatibility condition of $\mathcal M$ assures that of 
$\widetilde{\mathcal M}$.

We prepare a notation to give the definition of the middle convolution and analyze it.
\begin{definition}\label{def:iota}
For an integer $n$ greater than 1 and a positive integer $N$, we define
\begin{align*}
 L_n^0&=\{1,2,\dots,n{-}1\},\quad L_n=\{0,1,2,\dots,n{-}1\},\\
  {\iota_j(v)}&:={(v)_j}:=\text{\scriptsize $j$}\!\left(\begin{smallmatrix}0\\[-2mm] \vdots\\v\\[-2mm] \vdots\\ 0\end{smallmatrix}\right)\in\mathbb C^{(n-1)N},\quad  
\iota_j:\mathbb C^N\hookrightarrow\mathbb C^{(n-1)N}
\quad(v\in\mathbb C^N,\ j\in L_n^0),\\
 {\iota_I}&:=\sum_{i\in I}\iota_i:\mathbb C^N\hookrightarrow\mathbb C^{(n-1)N},\ 
 {(v)_I}:=\iota_I(v)\in\mathbb C^{(n-1)N}\quad(v\in\mathbb C^N,\ I\subset {L_n^0}),\\
V_I&:=\iota_I(\mathbb C^N)\simeq\mathbb C^N.
\end{align*}
For the equation $\mathcal M$ given by \eqref{eq:KZ}, we define subspaces of $\mathbb C^{(n-1)N}$:
\begin{align*}
 \mathcal K_i&:=\iota_i(\ker\, A_{0i})=
 \text{\scriptsize $i$}\!\left(\begin{smallmatrix}0\\[-2mm] \vdots\\ \!\ker A_{0i}\!\\[-2mm] \vdots\\ 0\end{smallmatrix}\right)\
{\subset\mathbb C^{(n-1)N}},\\
\mathcal K_\infty\,&:=\ker\tilde A_{0\infty}
\overset{\mu\ne0}{=}\iota_{L_n^0}\bigl(\ker 
(A_{0\infty}-\mu)\bigr)=
\left\{\left(\begin{smallmatrix}v\\[-2mm]\vdots\\v\end{smallmatrix}\right)\mid
A_{0\infty}v=\mu v\right\}
{\,\subset\mathbb C^{(n-1)N}},\\
\mathcal K\,&:=\mathcal K_\infty+\bigoplus_{i=1}^n
\mathcal K_i.
\end{align*}
\end{definition}
Here we remark that, if $\mu\ne0$, then 
$\ker\tilde A_{0\infty}=\iota_{L_n^0}\bigl(\ker(A_{0\infty}-\mu))$ and 
$\mathcal K$ is the direct sum of  $\mathcal K_1,\dots,\mathcal K_{n-1}$ and $\mathcal K_\infty$.

Since $\tilde A_I\mathcal K\subset\mathcal K$, $\tilde A_I$ induce linear transformations 
on the quotient space $\mathbb C^{(n-1)N}/\mathcal K$.
These linear transformation are expressed by matrices 
\begin{align}
{\bar A}_I\in M\bigl((n-1)N-\dim\mathcal K,\mathbb C\bigr)
\end{align}
under a base of the quotient space.
\begin{definition}[{\cite{DR},\ \cite{Ha}}]
The {\bf middle convolution} $\overline{\mathcal M}=\mc_{x_0,\mu}\mathcal M$ of $\mathcal M$
is defined by
\[
  \overline{\mathcal M}:\frac{\p \bar u}{\p x_i}=\sum_{\nu\in L_n\setminus\{i\}}\!\!\frac{\bar A_{i\nu}}
 {x_i-x_\nu}\bar u\quad(i\in L_n).
\]
\end{definition}
If the ordinary differential equation with the variable $x_0$ defined by $\mathcal M$
is irreducible and $\mu$ is generic, 
$\mc_{x_0,\mu}\mathcal M$ is also irreducible and it is proved by \cite{DR} that
\begin{equation}
 \mc_{x_0,\mu'}\circ\mc_{x_0,\mu}=\mc_{x_0,\mu+\mu'},\  \mc_{x_0,0}=\mathrm{id}.
\end{equation}
In most cases, the irreducibility of this ODE coincides with 
that of $\mathcal M$ (cf.~\cite{Oir}).


\medskip
When $n=4$, the convolution of $A_{123}=A_{12}+A_{13}+A_{23}$ is
\begin{align*}
\tilde A_{123}&=\left(\begin{matrix}
A_{12}+A_{02} & -A_{02}&0\\
-A_{01} & A_{12}+A_{01}&0\\
0 &0 &A_{12}
\end{matrix}\right)
+
\left(\begin{matrix}
A_{13}+A_{03} & 0 &-A_{03}\\
0      &  A_{13} &0      \\
-A_{01}&0 & A_{13}+A_{01}
\end{matrix}\right)\\
&\quad{}
+
\left(\begin{matrix}
A_{23}& 0 & 0\\
0      &  A_{23}+A_{03} &-A_{03}\\
0      &-A_{02} & A_{23}+A_{02}
\end{matrix}\right)\\
&=
\left(\begin{matrix}
A_{0123}-A_{01}&-A_{02} & -A_{03}\\
-A_{01}  &  A_{0123}-A_{02} &-A_{03}\\
-A_{01}  &-A_{02} & A_{0123}-A_{03}
\end{matrix}\right)\in M(3N,\mathbb C)
.
\end{align*}
Let $m$ be a positive integer smaller than $n$. Then
\begin{align*}
\tilde A_{1\cdots m}&
=\begin{pmatrix}
 A_{0\cdots m}{-}A_{01}\!\!&\cdots&\!-A_{0m}&&&\\
 \vdots&\ddots&\vdots\\
 \!\!-A_{01}&\cdots&\!A_{0\cdots m}{-}A_{0m}&&&\\
 &&&\!\!\!A_{1\cdots m}\!\!\!&&\\
 &&&&\ddots&\\
 &&&&&\!\!\!\!\!A_{1\cdots m}
\end{pmatrix}\in M((n-1)N)
.
\end{align*}
\begin{lemma}\label{lem:I}
Retain the notation in Definition~\ref{def:iota}.
For $I\subset L_n\ \ (|I|>1)$ and  $j,\,k\in L_n^0$ and $v\in\mathbb C^N$, 
we have
\begin{align*}
\tilde A_I(v)_J&=(A_Iv)_J\in V_J
&(I\subset J\subset L_n^0),\\ 
\tilde A_I(v)_j&=(A_{0I}v)_j-(A_{0j}v)_I
 &(I\ni j),\\
\tilde A_I(v)_k&=(A_Iv)_k\in V_j&(I\not\ni k),\\
[\tilde A_I]&=[A_I]_{n-|I|}\cup [A_{0I}]_{|I|-1},\\
\tilde A_I(v)_j&=(A_{0I}v)_j\in\mathcal K_j&(v\in\ker A_{0j},\ I\ni j),\\
\tilde A_I(v)_k&=(A_Iv)_k\in\mathcal K_k&(v\in\ker A_{0k},\ I\not\ni k),\\
\tilde A_I(v)_{L_n^0}&=(A_Iv)_{L_n^0}\in\mathcal K_\infty&(v\in\ker(A_{0\infty}-\mu)). 
\end{align*}
\end{lemma}
By the symmetry of coordinate $(x_1,\dots,x_n)$, we have only to prove this lemma
in the case $I=\{1,\dots,m\}$, but the lemma is clear by the above expression 
of $A_I$. Note that the last three equalities in the above follow from
the relation $[A_{0I},A_{0j}]=[A_I,A_{0k}]=[A_I,A_{0\infty}]=0$ \  $(j\in I,\ k\not\in I)$.

Since $\tilde A_{0\cdots m}=\tilde A_{1\cdots k}+\tilde A_{01}+\cdots+\tilde A_{0m}$, we have
\begin{align*}
\tilde A_{0\cdots m}&
=\begin{pmatrix}
A_{01\cdots m}+\mu&       &                        &A_{0,m+1} &\cdots&A_{0,n-1}\\
                       & \ddots&                        &\vdots &\cdots&\vdots\\
                       &       &A_{01 \cdots m}+\mu &A_{0,m+1} &\cdots&A_{0,n-1}\\
                                                      &&&A_{1\cdots m}   &      & \\
                                                      &&&       &\ddots&\\
                                                      &&&       && A_{1\cdots m}
\end{pmatrix}
.
\end{align*}
Then $\widetilde A_{0\dots n-1}$ is a block diagonal matrix with the diagonal
element $A_{0\dots n-1}+\mu$.
\begin{lemma}\label{lem:0I} {\rm (i)} \ 
For $I\subset L_n^0$ and $j,\,k \in L_n^0$ and $v\in\mathbb C^N$, we have
\begin{align*}
\tilde A_{0I}(v)_j&=((A_{0I}+\mu)v)_j\in V_{\{j\}}&(I\ni j),\\
\tilde A_{0I}(v)_k&=(A_Iv)_k+(A_{0k}v)_I
 &(I\not\ni k),\\
[\tilde A_{0I}]&=[A_{0I}+\mu]_{n-1-|I|}\cup [A_I]_{|I|},\\
\tilde A_{0I}(v)_j&=((A_{0I}+\mu)v)_j\in\mathcal K_j&(v\in\ker A_{0j},\ I\ni j),\\
\tilde A_{0I}(v)_k&=(A_{I}v)_k\in\mathcal K_k&(v\in\ker A_{0k},\ I\not\ni k),\\
\tilde A_{0I}(v)_{L_n^0}&=(A_{I}v)_{L_n^0}\in\mathcal K_\infty&(v\in\ker(A_{0\infty}-\mu)).
\end{align*}
\end{lemma}
Here the last equality in the above follows from the relation
\begin{equation}
(A_{0I}+\mu)v+\sum_{\nu\in L_n^0 \setminus I}A_{0\nu}v
 =A_Iv+\Bigr({\sum_{\nu=1}^{n-1}A_{0\nu}}+\mu\Bigl)v=A_Iv-(A_{0\infty}-\mu)v
.
\end{equation}

Let $\mathcal I$ be a maximal commuting family of $L_n$.
We denote by $\mathcal I_0$ the subset of $\mathcal I$ consisting the elements containing $0$.
The elements of $I_0$ are naturally ordered by the inclusion relationship and they are labelled 
as $I_{1,0}\subset I_{2,0}\subset\cdots\subset I_{m,0}=L_n$ with $m=|\mathcal I_||$.

Moreover we put $\mathcal I_k=\{I\in\mathcal I\mid I\subset I_{k,0}\setminus I_{k-1,0}\}$ and 
the elements of $\mathcal I_k$ are labelled as $I_{k,\nu}$ with $0<\nu \le m_k=|\mathcal I_k|$ so that 
$I_{k,\nu}\supset I_{k,\nu'}$ implies $I_{k,\nu}\le I_{k,\nu'}$.
\begin{definition}
We define labels and an order to the elements of 
a maximal commuting family $\mathcal I$ of $L_n=\{0,\dots,n{-}1\}$:
\begin{align*}
\mathcal I_0
   &:=\{I\in\mathcal I\mid 0\in I\}\\
   &=\bigl\{I_{k,0}\mid 1\le k\le m,\ I_{1,0}\subset I_{2,0}\subset \cdots \subset I_{m,0}\},
 \allowdisplaybreaks\\
\mathcal I_k&:=\{I\in\mathcal I\mid I\subset I_{k,0}\setminus I_{k-1,0}\},\\
&=\{I_{k,\nu}\mid \nu=1,\dots,m_k,\ I_{k,\nu}\supset I_{k,\nu'}\text{ or }
  I_{k,\nu}\cap I_{k,\nu'}=\emptyset
\quad(\nu\le\nu')\},\allowdisplaybreaks\\
I_{k,\nu}&\le I_{k',\nu'}\overset{\text{def,}}\Longleftrightarrow k<k' \text{ or }(k=k'\text{ and }\nu \le \nu'),\\
I^{(\ell)}&:=I_{k,\nu}
\text{ with }
\ell=|\{I\in\mathcal I\mid I\le I_{k,\nu}\}|\quad
 (1\le\ell<n).
\end{align*}
\end{definition}
Here the numbers $\ell$ indicate the order of the elements of $\mathcal I$. Namely,
\[
\mathcal I=\{I^{(\ell)}\mid \ell=1,\dots,n{-}1\},\quad I^{(1)}<I^{(2)}<\cdots<I^{(n-1)}.
\]
Thus a maximal commuting family $\mathcal I$ of $L_n$ becomes a totally ordered set.
Here we note that the order is not uniquely determined.
It may be easy to see this order if $\mathcal I$ is expressed by a figure 
of the corresponding tournament of the teams labeled by the elements of $L_n$ where the team 
with the label $0$ is the final winner.
\begin{exmp} 
In the case of the tournament \ 
\raisebox{-.2cm}{\begin{tikzpicture}
\draw[very thick]
(1.6,0)--(1.6,0.4)--(2.0,0.4)
(1.2,0.6)--(2,0.6)--(2,0.4)
(1.2,0.6)--(1.2,0.8)--(2.4,0.8)
(2.4,0.8)--(2.4,1.0)
;
\draw
(2.1,0.4)--(2.2,0.4)--(2.2,0.2)
(0.6,0.4)--(0.6,0.6)--(1.1,0.6)
(2.5,0.8)--(3.4,0.8)--(3.4,0.4)
(0,0)--(0,0.2)--(0.4,0.2)--(0.4,0)
(0.8,0)--(0.8,0.2)--(1.2,0.2)--(1.2,0)
(0.2,0.2)--(0.2,0.4)--(1,0.4)--(1,0.2)
(2,0)--(2,0.2)--(2.4,0.2)--(2.4,0)
(2.8,0)--(2.8,0.2)--(3.2,0.2)--(3.2,0)
(3.6,0)--(3.6,0.2)--(4,0.2)--(4,0)
(3,0.2)--(3,0.4)--(3.8,0.4)--(3.8,0.2)
node at (0,-0.2){$1$}
node at (0.4,-0.2){$2$}
node at (0.8,-0.2){$3$}
node at (1.2,-0.2){$4$}
node at (1.6,-0.2){$0$}
node at (2,-0.2){$5$}
node at (2.4,-0.2){$6$}
node at (2.8,-0.2){$7$}
node at (3.2,-0.2){$8$}
node at (3.6,-0.2){$9$}
node at (4,-0.2){$10$}
node at (2.4,1.1) {\small$\infty$}
;
\end{tikzpicture}}
,
\begin{align*}
\mathcal I_0&=\bigl\{I_{1,0}=\{0,5,6\},\ I_{2,0}=\{0,1,2,3,4,5,6\},
\ I_{3,0}=\{0,1,2,\ldots,9,10\}\bigr\},\\
\mathcal I_1&=\bigl\{I_{1,1}=\{5,6\}\bigr\},\\
\mathcal I_2&=\bigl\{I_{2,1}=\{1,2,3,4\},\  I_{2,2}=\{1,2\},\ I_{2,3}=\{3,4\}\bigr\},\\
\mathcal I_3&=\bigl\{I_{3,1}=\{7,8,9,10\},\ I_{3,2}=\{7,8\},\ I_{3,3}=\{9,10\}\bigr\},
\end{align*}
$I_{1,0}<I_{1,1}<I_{2,0}<I_{2,1}<I_{2,2}<I_{2,3}<I_{3,0}<I_{3,1}<I_{3,2}<I_{3,3}.$
\end{exmp}

Now we recall the map $b^0$ in Definition~\ref{def:bry}. 
The set of teams $b^0(I)$ is not necessarily uniquely determined by a tournament.
But if we identify $b^0(I)$ with the teams failed the game labeled by $I$, 
the definition of $b^0$ corresponds to the tournament figure 
where the winner of each game is indicated.  An example of the figure is
\begin{equation}
\raisebox{-2mm}{\begin{tikzpicture}
\draw[very thick]
(0,0.2)--(0.1,0.2)
(0.8,0.2)--(0.9,0.2)
(2,0.2)--(2.1,0.2)
(2.8,0.2)--(2.9,0.2)
(3.6,0.2)--(3.7,0.2)
(0.2,0.4)--(0.5,0.4)
(2.1,0.4)--(2.2,0.4)
(3.0,0.4)--(3.3,0.4)
(0.6,0.6)--(1.1,0.6)
(2.5,0.8)--(3.4,0.8)
;
\draw
(0.2,0.2)--(0.4,0.2)
(3.8,0.2)--(4,0.2)
(0.6,0.4)--(1,0.4)
(3.4,0.4)--(3.8,0.4)
;
\draw
(1.6,0)--(1.6,0.4)--(2.0,0.4)
(1.2,0.6)--(2,0.6)--(2,0.4)
(1.2,0.6)--(1.2,0.8)--(2.4,0.8)
(2.4,0.8)--(2.4,1.0)
;
\draw
(0,0)--(0,0.2)--(0.1,0.2) (0.4,0.2)--(0.4,0)
(0.8,0)--(0.8,0.2) (1.0,0.2)--(1.2,0.2)--(1.2,0)
(2,0)--(2,0.2) (2.2,0.2)--(2.4,0.2)--(2.4,0)
(2.8,0)--(2.8,0.2) (3.0,0.2)--(3.2,0.2)--(3.2,0)
(3.6,0)--(3.6,0.2)--(3.7,0.2) (4,0.2)--(4,0)
(0.2,0.2)--(0.2,0.4)--(0.5,0.4) (1,0.4)--(1,0.2)
(2.2,0.4)--(2.2,0.2)
(3,0.2)--(3,0.4)--(3.3,0.4) (3.8,0.4)--(3.8,0.2)
(0.6,0.4)--(0.6,0.6)--(1.1,0.6)
(2.5,0.8)--(3.4,0.8)--(3.4,0.4)
node at (0,-0.2){$1$}
node at (0.4,-0.2){$2$}
node at (0.8,-0.2){$3$}
node at (1.2,-0.2){$4$}
node at (1.6,-0.2){$0$}
node at (2,-0.2){$5$}
node at (2.4,-0.2){$6$}
node at (2.8,-0.2){$7$}
node at (3.2,-0.2){$8$}
node at (3.6,-0.2){$9$}
node at (4,-0.2){$10$}
node at (2.4,1.1) {\small$\infty$}
;
\end{tikzpicture}}
\end{equation}

Here $n=11$ and the corresponding $b^0(I^{(1)}),b^0(I^{(2)}),\ldots, b^0(I^{(n-1)})$ are 
\begin{equation*}
\underset{6}{\overset{\{0,5,6\}}{\{5,6\}}},\,
\underset{5}{\overset{\{5,6\}}{\{5\}}},\,
\underset{4}{\overset{\{0,1,2,3,4,5,6\}}{\{1,2,3,4\}}},\,
\underset{2}{\overset{\{1,2,3,4\}}{\{1,2\}}},\,
\underset{1}{\overset{\{1,2\}}{\{1\}}},\,
\underset{3}{\overset{\{3,4\}}{\{3\}}},\,
\underset{10}{\overset{\{0,1,\dots,10\}}{\{7,8,9,10\}}},
\underset{8}{\overset{\{7,8,9,10\}}{\{7,8\}}},\,
\underset{7}{\overset{\{7,8\}}{\{7\}}},\,
\underset{9}{\overset{\{9,10\}}{\{9\}}}
\end{equation*}
and the number under the set $b^0(I^{(i)})$ means that the corresponding 
team failed in the game $I^{(i)}$.
\medskip
\begin{definition}
Define subspaces of $\mathbb C^{(n-1)N}$ as follows (cf.~Definitions~\ref{def:bry}, 
\ref{def:iota}).
\begin{align*}
  W^{(\ell)}&:=W_I:= \sum_{\nu=1}^\ell 
    V_{b^0(I^{(\nu)})} \quad(I=I^{(\ell)}\in\mathcal I),\\
  U_{b^0(\mathcal I)}&:=\bigl(U_{ij}\bigr)_{\substack{1\le i<n-1\\ 1\le j\le n-1}}
    \in GL((n-1)N,\mathbb C)\quad\text{with}\\
  &\qquad U_{ij}=\begin{cases}
           E_N&(j\in b^0(I^{(i)})),\\
           O_N&(j\not\in b^0(I^{(i)})).
           \end{cases}
\end{align*}
Here $E_N$ is the identity matrix and $O_N$ is
the zero matrix in $M(N,\mathbb C)$.
\end{definition}

\begin{lemma}\label{lem:main}
{\rm (i)} \ 
For $0\not\in I\in\mathcal I$, we have
\begin{align}
 V_I&\subset W_I. \label{eq:W1}\\
 V_{\bar b(I)},\, V_{\bar b'(I)}&\subset W_I.\label{eq:W2}
\end{align}

{\rm (ii)} \ $\dim W^{(\ell)}=\ell N$ \quad$(\ell=1,\dots,n{-}1)$

{\rm (iii)} \ 
Let $I,\,K\in \mathcal I$ and $v\in\mathbb C^N$. Putting $J=b^0(K)$,
we have the following.
\begin{align*}
\intertext{If $0\not\in I$, then}
  \tilde A_I (v)_J&=
   \begin{cases}
       (A_I v)_J&(I\not\supset K),\\
       (A_{0I} v)_J - \bigl(\sum\limits_{\nu\in J}A_{0\nu}v\bigr)_I\phantom{+++11\,}&(I\supset K).
   \end{cases}
\intertext{If $0\in I$, then} 
  \tilde A_I (v)_J&=
   \begin{cases}
       \bigl((A_I+\mu) v\bigr)_J&(I\supset K),\\
       (A_{I\setminus\{0\}} v)_J+\bigl(\sum\limits_{\nu\in J}A_{0\nu}v\bigr)_{I\setminus\{0\}}
          &(I\not \supset K).
   \end{cases}
\end{align*}

{\rm (iv)} \ $\tilde A_I W^{(\ell)}\subset W^{(\ell)}\quad (I\in \mathcal I,\ 1\le \ell<n)$. 
\end{lemma}
\begin{proof}
Putting $I=I_{k,i}$, we will first prove \eqref{eq:W1} and \eqref{eq:W2} 
by the induction with respect to $i$.
Since $I=\bar b(I)\sqcup\bar b'(I)$, we remark that \eqref{eq:W2} follows from \eqref{eq:W1}.
When $i=1$, \eqref{eq:W1} is clear because $b^0(I_{k,0})=I_{k,1}$.
When $i>1$, there exists $J\in \mathcal I$ satisfying $I\subset J\in\mathcal I_k$ and 
$I\in b(J)$ (cf.~\eqref{eq:bI}). Then $V_I\subset W_J\subset W_I$ by the induction hypothesis for \eqref{eq:W2}
with replacing $I$ by $J$.

Hence
\begin{align*}
 \sum_{k=1}^m\sum_{\nu=0}^{m_k} V_{b^0(I_{k,\nu})}
 &=\sum_{k=1}^m\sum_{\nu=0}^{m_k} W_{I_{k,\nu}}
 =\sum_{k=1}^m\Bigl(V_{b^0(I_{k,0})}
 +\sum_{\nu=1}^{m_k}(V_{\bar b(I_{k,i})}+V_{\bar b'(I_{k,i})})\Bigr)\\
 &
 \supset \sum_{\nu=1}^{n-1} V_{\{\nu\}}\simeq\mathbb C^{(n-1)N}
.
\end{align*}
Since $\sum\limits_{k=1}^m(m_k+1)=n-1$, 
$\sum\limits_{k=1}^m\sum\limits_{\nu=0}^{m_k} V_{b^0(I_{k,\nu})}$ is a direct sum
decomposition of $\mathbb C^{(n-1)N}$, which implies (ii).

Then Lemma~\ref{lem:I} and  Lemma~\ref{lem:0I} show (iii).
In particular we have (iv). 
\end{proof}

\begin{definition}\label{def:md}
Let $L$ be a finite set.
For $i\in L$ and nonempty subsets $I$ and $J$ of $ L$, we define
\[
 \md_{i,J}(I):=\begin{cases}
  I\cup\{i\}      &(I\supset J),\\
  I\setminus\{i\} &(I\not\supset J),
\end{cases}\quad
 \me_{i,J}(I):=\begin{cases}
  1 &(i\in I\supset J),\\
  0 &(i\not\in I\text{ or }I\not\supset J).
\end{cases}
\]
\end{definition}
\begin{remark} (i) \ If \ $i\in I\supset J$\ or \ $i\not\in I\not\supset J$, then 
\ $\md_{i,J}(I)=I$.

(ii) Let $K=I^{(\ell)}\in\mathcal I$.  Then $\widetilde A_I$ induces a linear transformation
on the quotient space $W^{(\ell)}/W^{(\ell-1)}\simeq V_{b^0(K)}\simeq \mathbb C^N$, which
is identified with $A_I^K$ given by \eqref{eq:AJK}.

(iii) \ If the elements of a family $\{I_\nu\mid\nu=0,1,\dots,r\}$ of subsets of $L$
commute with each other (cf.~\eqref{def:ComSets}) and $|I_\nu|>1$ $(\nu=1,\dots,r)$, 
so do the elements of $\{I_0\cup\{i\}\}\cup\{\md_{i,I_0}(I_\nu)\mid(\nu=1,\dots,r)\}$.
\end{remark}

\medskip
We have the following theorem from 
Lemma~\ref{lem:I}, Lemma~\ref{lem:0I} and Lemma~\ref{lem:main}.

\begin{theorem}\label{thm:main} 
{\rm (i)} \ Let $\mathcal I$ be a maximal commuting family of $L_n=\{0,1,\dots,n-1\}$.
Put $\mathcal I=\{I^{(1)},\ldots,I^{(n-1)}\}$.
Let $I\in\mathcal I$. 
Then under the notation in Definition~\ref{def:md}, we have
\begin{align*}
[\tilde A_{I}]&=
     [A_{I\cup\{0\}}+\mu]_{|I|-1} \sqcup [A_{I\setminus\{0\}}]_{n-|I|}
\allowdisplaybreaks,\\
[\tilde A_{I^{(1)}}:\cdots:\tilde A_{I^{(n-1)}}]&=
\bigsqcup_{J\in\mathcal I}[A_{I^{(1)}}^J:\dots:A_{I^{(n-1)}}^J],\\
[\tilde A_{I^{(1)}}:\cdots:\tilde A_{I^{(n-1)}}]|_{\mathcal K_j}
  &=[A_{I^{(1)}}^{\{j\}}:\cdots:A_{I^{(n-1)}}^{\{j\}}]|_{\ker A_{0j}}
   \quad(j\in L_n^0),\\
[\tilde A_{I^{(1)}}:\cdots:\tilde A_{I^{(n-1)}}]|_{\mathcal K_\infty}
&=[A_{I^{(1)}}^{L_n}:\cdots:A_{I^{(n-1)}}^{L_n}]|_{\ker(A_{0\infty}-\mu)}
\end{align*}
with denoting 
\begin{equation}\label{eq:AJK}
A_I^K := A_{\md_{0,K}(I)}+\me_{0,K}(I)\cdot\mu.
\end{equation}

{\rm (ii)} \ 
By a conjugation, 
$\widetilde A_{I^{(i)}}$ are simultaneously changed into 
upper triangular block matrices 
$U^{-1}_{b^0(\mathcal I)}\widetilde A_{I^{(i)}}U_{b^0(\mathcal I)}$. 
\end{theorem}

\begin{remark}\label{rem:omit}
(i) \ 
In the above theorem, we have
\begin{equation}
 A_{L_n}^J=A_{0\cdots n-1}^j=\mu+A_{0\cdots n-1}\text{ \ and  \ }
 A_{L_n}^{L_n}=A_{L_n^0}
\end{equation}
and therefore we often omit the term $\tilde A_{L_n}$ in $[\tilde A_{I^{(1)}}:\cdots:\tilde A_{I^{(n-1)}}]$.

Moreover we remark \[A_{0\infty}=A_{L_n^0}-A_{L_n}\] and Corollary~\ref{cor:kappa} and 
\begin{equation}
  \mc_{x_0,\mu} = \Ad\bigl((x_p-x_q)^{-\lambda}\bigr)\circ \mc_{x_0,\mu}\circ 
  \Ad\bigl((x_p-x_q)^\lambda\bigr)\quad(1\le p<q\le n{-}1).
\end{equation}

(ii) \ 
The generalized Riemann scheme
\begin{equation}
  \bigl\{[\bar A_{ij}]\mid \{i,j\}\subset \tilde L_n\bigr\}
\end{equation}
of the residue matrices $\mathrm{mc}_{x_0,\mu}\mathcal M$ is obtained 
by Theorem~\ref{thm:main} and $\Sp\mathcal M$.
\end{remark}

We examine the relation between the theorem and tournaments.
We write $\mathcal I$ by the corresponding tournament and the team with label $i$
by $(i)$. The theorem gives the transformation of $\Sp\mathcal M$ using $A_I^K$ in
\eqref{eq:AJK} for $I\in\mathcal I$. 
The corresponding transformation $\md_{x_0,I}^K$ of the tournament $\mathcal I$ 
is an insertion of (0) after the deletion of (0). The deleted game is replaced by 
the preceding game.  The inserted game which is the first game of (0) 
is as follows.
\begin{itemize}
\item
$(A_I^J)_{I\in\mathcal I}$ : the new game is played with the winner of 
the game $J$ and the preceding game of the opponent is replaced by the new game. 

\item
$(A_I^j|_{\ker A_{0j}})_{I\in\mathcal I}$ : the new game is played with ($j$) 
as a basic insertion.
\item
$(A_I^\infty|_{\ker(A_{0\infty}-\mu)})_{I\in\mathcal I}$ : the new game is played with
the winner of the tournament as a top insertion.  Here we put $A_I^\infty=A_I^{L_n}$.
\end{itemize}
We show the above by diagrams using examples: 
\\[-12mm]

\[
\begin{tikzpicture}
\draw[densely dotted]
(1.6,0)--(1.6,0.4)--(2.0,0.4)
;
\draw
(1.2,0.6)--(2,0.6)--(2,0.4)
(1.2,0.6)--(1.2,0.8)--(2.4,0.8)
(2.4,0.8)--(2.4,1.0)
(2.0,0.4)--(2.2,0.4)--(2.2,0.2)
(0.6,0.4)--(0.6,0.6)--(1.2,0.6)
(2.4,0.8)--(3.4,0.8)--(3.4,0.4)
(0,0)--(0,0.2)--(0.4,0.2)--(0.4,0)
(0.8,0)--(0.8,0.2)--(1.2,0.2)--(1.2,0)
(0.2,0.2)--(0.2,0.4)--(1,0.4)--(1,0.2)
(2,0)--(2,0.2)--(2.4,0.2)--(2.4,0)
(2.8,0)--(2.8,0.2)--(3.2,0.2)--(3.2,0)
(3.6,0)--(3.6,0.2)--(4,0.2)--(4,0)
(3,0.2)--(3,0.4)--(3.8,0.4)--(3.8,0.2)
node at (0,-0.2){$1$}
node at (0.4,-0.2){$2$}
node at (0.8,-0.2){$3$}
node at (1.2,-0.2){$4$}
node at (1.6,-0.2){$0$}
node at (2,-0.2){$5$}
node at (2.4,-0.2){$6$}
node at (2.8,-0.2){$7$}
node at (3.2,-0.2){$8$}
node at (3.6,-0.2){$9$}
node at (4,-0.2){$10$}
node at (2.4,1.1) {\small$\infty$}
node at (1,0.3) {$\bullet$}
node at (2,0.4) {$\bullet$}
node at (1.35,0.3) {$\leftarrow$}
;
\draw
(0.2,0.2)--(0.2,0.4)--(1,0.4)
(2,0.4)--(2.2,0.4)--(2.2,0.2)
;
\end{tikzpicture}
\quad\raisebox{4mm}{$\xrightarrow[J=\{3,4\}]{\md_{0,J}}$}\quad
\begin{tikzpicture}
\draw[densely dotted]
(0.8,0.2)--(1.2,0.2);
\draw[very thick]
(1,0.2)--(1,0.4)--(1.6,0.4)--(1.6,0);
\draw
(0.8,0)--(0.8,0.2)  (1.2,0.2)--(1.2,0)
(0,0)--(0,0.2)--(0.4,0.2)--(0.4,0)
(2,0)--(2,0.2)--(2.4,0.2)--(2.4,0)
(2.8,0)--(2.8,0.2)--(3.2,0.2)--(3.2,0)
(3.6,0)--(3.6,0.2)--(4,0.2)--(4,0)
(3,0.2)--(3,0.4)--(3.8,0.4)--(3.8,0.2)
(0.2,0.2)--(0.2,0.6)--(1.2,0.6)--(1.2,0.4)
(0.8,0.6)--(0.8,0.8)--(2.2,0.8)--(2.2,0.2)
(1.2,0.8)--(1.2,1)--(3.4,1)--(3.4,0.4)
(2,1)--(2,1.2)
node at (0,-0.2){$1$}
node at (0.4,-0.2){$2$}
node at (0.8,-0.2){$3$}
node at (1.2,-0.2){$4$}
node at (1.6,-0.2){$0$}
node at (2,-0.2){$5$}
node at (2.4,-0.2){$6$}
node at (2.8,-0.2){$7$}
node at (3.2,-0.2){$8$}
node at (3.6,-0.2){$9$}
node at (4,-0.2){$10$}
node at (2.2,0.4){$\bullet$}
node at (2,1.3) {\small$\infty$}
;
\end{tikzpicture}\\[-3.7mm]
\]
$J={\{3,4\}}:\{3,4\}\to\{0,3,4\},\ \{1,2,3,4\}\to\{0,1,2,3,4\},\ 
\{0,5,6\}\to\{5,6\}$

\[
%
\begin{tikzpicture}
\draw[densely dotted]
(1.6,0)--(1.6,0.4)--(2.0,0.4)
;
\draw
(1.2,0.6)--(2,0.6)--(2,0.4)
(1.2,0.6)--(1.2,0.8)--(2.4,0.8)
(2.4,0.8)--(2.4,1.0)
(2.0,0.4)--(2.2,0.4)--(2.2,0.2)
(0.6,0.4)--(0.6,0.6)--(1.2,0.6)
(2.4,0.8)--(3.4,0.8)--(3.4,0.4)
(0,0)--(0,0.2)--(0.4,0.2)--(0.4,0)
(0.8,0)--(0.8,0.2)--(1.2,0.2)--(1.2,0)
(0.2,0.2)--(0.2,0.4)--(1,0.4)--(1,0.2)
(2,0)--(2,0.2)--(2.4,0.2)--(2.4,0)
(2.8,0)--(2.8,0.2)--(3.2,0.2)--(3.2,0)
(3.6,0)--(3.6,0.2)--(4,0.2)--(4,0)
(3,0.2)--(3,0.4)--(3.8,0.4)--(3.8,0.2)
node at (0,-0.2){$1$}
node at (0.4,-0.2){$2$}
node at (0.8,-0.2){$3$}
node at (1.2,-0.2){$4$}
node at (1.6,-0.2){$0$}
node at (2,-0.2){$5$}
node at (2.4,-0.2){$6$}
node at (2.8,-0.2){$7$}
node at (3.2,-0.2){$8$}
node at (3.6,-0.2){$9$}
node at (4,-0.2){$10$}
node at (2.4,1.1) {\small$\infty$}
node at (2,0.4) {$\bullet$}
node at (1.2,0.7) {$\bullet$}
node at (1.35,0.4) {$\nwarrow$}
;
\draw
(0.2,0.2)--(0.2,0.4)--(1,0.4)
(2,0.4)--(2.2,0.4)--(2.2,0.2)
;
\end{tikzpicture}
\quad\raisebox{4mm}{$\xrightarrow[J=\{0,1,2,3,4,5,6\}]{\md_{0,J}}$}\quad
\begin{tikzpicture}
\draw[very thick]
(0,0)--(0,0.8)--(1.4,0.8)--(1.4,0.6);
\draw[densely dotted]
(1,0.6)--(2.2,0.6);
\draw
(0.4,0)--(0.4,0.2)--(0.8,0.2)--(0.8,0)
(1.2,0)--(1.2,0.2)--(1.6,0.2)--(1.6,0)
(2,0)--(2,0.2)--(2.4,0.2)--(2.4,0)
(2.8,0)--(2.8,0.2)--(3.2,0.2)--(3.2,0)
(3.6,0)--(3.6,0.2)--(4,0.2)--(4,0)
(0.6,0.2)--(0.6,0.4)--(1.4,0.4)--(1.4,0.2)
(3,0.2)--(3,0.4)--(3.8,0.4)--(3.8,0.2)
(1,0.4)--(1,0.6) (2.2,0.6)--(2.2,0.2)
(1.2,0.8)--(1.2,1)--(3.4,1)--(3.4,0.4)
(2,1)--(2,1.2)
node at (0,-0.2){$0$}
node at (0.4,-0.2){$1$}
node at (0.8,-0.2){$2$}
node at (1.2,-0.2){$3$}
node at (1.6,-0.2){$4$}
node at (2,-0.2){$5$}
node at (2.4,-0.2){$6$}
node at (2.8,-0.2){$7$}
node at (3.2,-0.2){$8$}
node at (3.6,-0.2){$9$}
node at (4,-0.2){$10$}
node at (2,1.3) {\small$\infty$}
node at (2.2,0.4){$\bullet$}
;
\end{tikzpicture}\\[-3.5mm]
\]
$J={\{0,1,2,3,4,5,6\}}:\{0,5,6\}\to\{5,6\}$

%
\[
\begin{tikzpicture}
\draw[densely dotted]
(1.6,0)--(1.6,0.4)--(2.0,0.4)
;
\draw
(0.8,0)--(0.8,0.2)--(1.2,0.2)--(1.2,0)
;
\draw
(1.2,0.6)--(2,0.6)--(2,0.4)
(1.2,0.6)--(1.2,0.8)--(2.4,0.8)
(2.4,0.8)--(2.4,1.0)
(2.0,0.4)--(2.2,0.4)--(2.2,0.2)
(0.6,0.4)--(0.6,0.6)--(1.2,0.6)
(2.4,0.8)--(3.4,0.8)--(3.4,0.4)
(0,0)--(0,0.2)--(0.4,0.2)--(0.4,0)
(0.8,0)--(0.8,0.2)--(1.2,0.2)--(1.2,0)
(0.2,0.2)--(0.2,0.4)--(1,0.4)--(1,0.2)
(2,0)--(2,0.2)--(2.4,0.2)--(2.4,0)
(2.8,0)--(2.8,0.2)--(3.2,0.2)--(3.2,0)
(3.6,0)--(3.6,0.2)--(4,0.2)--(4,0)
(3,0.2)--(3,0.4)--(3.8,0.4)--(3.8,0.2)
node at (0,-0.2){$1$}
node at (0.4,-0.2){$2$}
node at (0.8,-0.2){$3$}
node at (1.2,-0.2){$4$}
node at (1.6,-0.2){$0$}
node at (2,-0.2){$5$}
node at (2.4,-0.2){$6$}
node at (2.8,-0.2){$7$}
node at (3.2,-0.2){$8$}
node at (3.6,-0.2){$9$}
node at (4,-0.2){$10$}
node at (2.4,1.1) {\small$\infty$}
node at (0.8,0.1) {$\bullet$}
node at (2,0.4) {$\bullet$}
node at (1.45,0.1) {\small$\leftarrow$}
;
\draw
(0.2,0.2)--(0.2,0.4)--(1,0.4)
(2,0.4)--(2.2,0.4)--(2.2,0.2)
;
\end{tikzpicture}
\qquad\raisebox{4.5mm}{$\xrightarrow[j=3]{\md_{0,\{j\}}}$}\qquad
\begin{tikzpicture}
\draw[very thick]
(0.8,0)--(0.8,0.2)--(1.2,0.2)--(1.2,0);
\draw
(0,0)--(0,0.2)--(0.4,0.2)--(0.4,0)
(2,0)--(2,0.2)--(2.4,0.2)--(2.4,0)
(2.8,0)--(2.8,0.2)--(3.2,0.2)--(3.2,0)
(3.6,0)--(3.6,0.2)--(4,0.2)--(4,0)
(1,0.2)--(1,0.4)--(1.6,0.4)--(1.6,0)
(3,0.2)--(3,0.4)--(3.8,0.4)--(3.8,0.2)
(0.2,0.2)--(0.2,0.6)--(1.2,0.6)--(1.2,0.4)
(0.8,0.6)--(0.8,0.8)--(2.2,0.8)--(2.2,0.2)
(1.2,0.8)--(1.2,1)--(3.4,1)--(3.4,0.4)
(2,1)--(2,1.2)
node at (0,-0.2){$1$}
node at (0.4,-0.2){$2$}
node at (0.8,-0.2){$0$}
node at (1.2,-0.2){$3$}
node at (1.6,-0.2){$4$}
node at (2,-0.2){$5$}
node at (2.4,-0.2){$6$}
node at (2.8,-0.2){$7$}
node at (3.2,-0.2){$8$}
node at (3.6,-0.2){$9$}
node at (4,-0.2){$10$}
node at (2.2,0.4){$\bullet$}
node at (2,1.3) {\small$\infty$}
;
\end{tikzpicture}\\[-3.5mm]
\]
$j=3: \{3,4\}\to\{0,3,4\},\ \{1,2,3,4\}\to\{0,1,2,3,4\},\ \{0,5,6\}\to\{5,6\},\ 
+\{0,3\}$ 

\[
\begin{tikzpicture}
\draw[densely dotted]
(1.6,0)--(1.6,0.4)--(2.0,0.4)
;
\draw
(0.8,0)--(0.8,0.2)--(1.2,0.2)--(1.2,0)
;
\draw
(1.2,0.6)--(2,0.6)--(2,0.4)
(1.2,0.6)--(1.2,0.8)--(2.4,0.8)
(2.4,0.8)--(2.4,1.05)
(2.0,0.4)--(2.2,0.4)--(2.2,0.2)
(0.6,0.4)--(0.6,0.6)--(1.2,0.6)
(2.4,0.8)--(3.4,0.8)--(3.4,0.4)
(0,0)--(0,0.2)--(0.4,0.2)--(0.4,0)
(0.8,0)--(0.8,0.2)--(1.2,0.2)--(1.2,0)
(0.2,0.2)--(0.2,0.4)--(1,0.4)--(1,0.2)
(2,0)--(2,0.2)--(2.4,0.2)--(2.4,0)
(2.8,0)--(2.8,0.2)--(3.2,0.2)--(3.2,0)
(3.6,0)--(3.6,0.2)--(4,0.2)--(4,0)
(3,0.2)--(3,0.4)--(3.8,0.4)--(3.8,0.2)
node at (0,-0.2){$1$}
node at (0.4,-0.2){$2$}
node at (0.8,-0.2){$3$}
node at (1.2,-0.2){$4$}
node at (1.6,-0.2){$0$}
node at (2,-0.2){$5$}
node at (2.4,-0.2){$6$}
node at (2.8,-0.2){$7$}
node at (3.2,-0.2){$8$}
node at (3.6,-0.2){$9$}
node at (4,-0.2){$10$}
node at (2.4,1.15) {\small$\infty$}
node at (2,0.4) {$\bullet$}
node at (2.4,0.9){$\bullet$}
node at (1.7,0.2) {$\uparrow$}
;
\draw
(0.2,0.2)--(0.2,0.4)--(1,0.4)
(2,0.4)--(2.2,0.4)--(2.2,0.2)
;
\end{tikzpicture}
\qquad\raisebox{4mm}{$\xrightarrow[j=\infty]{\md_{0,L_n}}$}\qquad
\begin{tikzpicture}
\draw[very thick]
(0,0)--(0,1)--(2.2,1)--(2.2,0.8)
;
\draw
(0.4,0)--(0.4,0.2)--(0.8,0.2)--(0.8,0)
(1.2,0)--(1.2,0.2)--(1.6,0.2)--(1.6,0)
(2,0)--(2,0.2)--(2.4,0.2)--(2.4,0)
(2.8,0)--(2.8,0.2)--(3.2,0.2)--(3.2,0)
(3.6,0)--(3.6,0.2)--(4,0.2)--(4,0)
(0.6,0.2)--(0.6,0.4)--(1.4,0.4)--(1.4,0.2)
(3,0.2)--(3,0.4)--(3.8,0.4)--(3.8,0.2)
(1,0.4)--(1,0.6)--(2.2,0.6)--(2.2,0.2)
(1.4,0.6)--(1.4,0.8)--(3.4,0.8)--(3.4,0.4)
(1.4,1)--(1.4,1.2)
node at (0,-0.2){$0$}
node at (0.4,-0.2){$1$}
node at (0.8,-0.2){$2$}
node at (1.2,-0.2){$3$}
node at (1.6,-0.2){$4$}
node at (2,-0.2){$5$}
node at (2.4,-0.2){$6$}
node at (2.8,-0.2){$7$}
node at (3.2,-0.2){$8$}
node at (3.6,-0.2){$9$}
node at (4,-0.2){$10$}
node at (2.2,0.4){$\bullet$}
node at (1.4,1.3) {\small$\infty$}
;
\end{tikzpicture}\\[-3.5mm]
\]
$j=\infty: \{0,5,6\}\!\to\!\{5,6\},\,\{0,1,\ldots,6\}\to\{1,\ldots,6\},\,\{0,1,\dots,10\}\!\to\!\{1,\dots,10\}$ 

\begin{remark} \label{rem:spectra}
Middle convolutions $\mc_{x_0,\mu}$, additions $\Ad\bigl((x_i-x_j)^\lambda \bigr)$ and
permutations of suffices $\{0,1,\dots,n{-}1\}$ define transformations on the space of KZ-type equations.
The change of spectra $\Sp\mathcal M$ is obtained by Theorem~\ref{thm:main}.

We examine the necessary data of $\mathcal M$ 
to get $[A_{12}]$ of a equation which is obtained by
a successive application of these transformation to the original KZ-type equation $\mathcal M$.
Since $[\tilde A_{12}|_{\mathcal K_1}]=[A_{012}|_{\ker A_{01}}]$, we need $[A_{01}:A_{012}]$
considering additions.
Considering more permutations, we need $[A_{12}:A_{123}]$ and 
$[A_{23}:A_{123}]$ in general.
Considering a middle convolution with respect to $x_0$, we moreover need
$[A_{01}:A_{012}:A_{0123}]$ and $[A_{01}:A_{23}:A_{0123}]$.
These considerations correspond to the following diagram.

\noindent
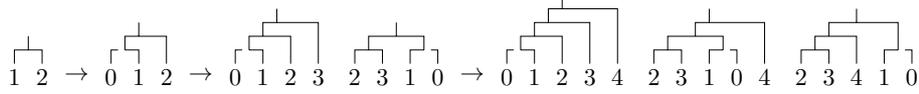
\scalebox{0.91}{\begin{tikzpicture}
\draw
(0,0)--(0,0.2)--(0.4,0.2)--(0.4,0)
(0.2,0.2)--(0.2,0.4)
node at (0,-0.2){$1$}
node at (0.4,-0.2){$2$}
;
\end{tikzpicture}
\raisebox{2mm}{$\to$}
\begin{tikzpicture}
\draw
(0,0)--(0,0.2)--(0.1,0.2) (0.2,0.2)--(0.4,0.2)--(0.4,0)
(0.2,0.2)--(0.2,0.4)--(0.8,0.4)--(0.8,0)
(0.4,0.4)--(0.4,0.6)
node at (0,-0.2){$0$}
node at (0.4,-0.2){$1$}
node at (0.8,-0.2){$2$}
;
\end{tikzpicture}
\raisebox{2mm}{$\to$}
\begin{tikzpicture}
\draw
(0,0)--(0,0.2)--(0.1,0.2) (0.2,0.2)--(0.4,0.2)--(0.4,0)
(0.2,0.2)--(0.2,0.4)--(0.8,0.4)--(0.8,0)
(0.4,0.4)--(0.4,0.6)--(1.2,0.6)--(1.2,0)
(0.6,0.6)--(0.6,0.8)
node at (0,-0.2){$0$}
node at (0.4,-0.2){$1$}
node at (0.8,-0.2){$2$}
node at (1.2,-0.2){$3$}
;
\end{tikzpicture}
\begin{tikzpicture}
\draw
(0,0)--(0,0.2)--(0.4,0.2)--(0.4,0)
(0.8,0)--(0.8,0.2)--(1.0,0.2) (1.1,0.2)--(1.2,0.2)--(1.2,0)
(0.2,0.2)--(0.2,0.4)--(1,0.4)--(1,0.2)
(0.6,0.4)--(0.6,0.6)
node at (0,-0.2){$2$}
node at (0.4,-0.2){$3$}
node at (0.8,-0.2){$1$}
node at (1.2,-0.2){$0$}
;
\end{tikzpicture}
\raisebox{2mm}{$\to$}
\begin{tikzpicture}
\draw
(0,0)--(0,0.2)--(0.1,0.2) (0.2,0.2)--(0.4,0.2)--(0.4,0)
(0.2,0.2)--(0.2,0.4)--(0.8,0.4)--(0.8,0)
(0.4,0.4)--(0.4,0.6)--(1.2,0.6)--(1.2,0)
(0.6,0.6)--(0.6,0.8)--(1.6,0.8)--(1.6,0)
(0.8,0.8)--(0.8,1)
node at (0,-0.2){$0$}
node at (0.4,-0.2){$1$}
node at (0.8,-0.2){$2$}
node at (1.2,-0.2){$3$}
node at (1.6,-0.2){$4$}
;
\end{tikzpicture}
\begin{tikzpicture}
\draw
(0,0)--(0,0.2)--(0.4,0.2)--(0.4,0)
(0.8,0)--(0.8,0.2)--(1.0,0.2) (1.1,0.2)--(1.2,0.2)--(1.2,0)
(0.2,0.2)--(0.2,0.4)--(1,0.4)--(1,0.2)
(0.6,0.4)--(0.6,0.6)--(1.6,0.6)--(1.6,0)
(0.8,0.6)--(0.8,0.8)
node at (0,-0.2){$2$}
node at (0.4,-0.2){$3$}
node at (0.8,-0.2){$1$}
node at (1.2,-0.2){$0$}
node at (1.6,-0.2){$4$}
;
\end{tikzpicture}
\begin{tikzpicture}
\draw
(0,0)--(0,0.2)--(0.4,0.2)--(0.4,0)
(1.2,0)--(1.2,0.2)--(1.4,0.2) (1.5,0.2)--(1.6,0.2)--(1.6,0)
(0.2,0.2)--(0.2,0.4)--(0.8,0.4)--(0.8,0)
(0.4,0.4)--(0.4,0.6)--(1.4,0.6)--(1.4,0.2)
(0.8,0.6)--(0.8,0.8)
node at (0,-0.2){$2$}
node at (0.4,-0.2){$3$}
node at (0.8,-0.2){$4$}
node at (1.2,-0.2){$1$}
node at (1.6,-0.2){$0$}
;
\end{tikzpicture}}

These changes of patterns correspond to basic insertions.
Since any tournament is obtained by a successive application of basic insertions 
(cf.~Remark~\ref{rem:basic-add}), we need $\Sp\mathcal M$ in general.
 
On the other hand, convolutions, additions and permutations of suffixes define transformation
on the space $\{[A_I]\mid I \subset L_n\}$ which does not contain simultaneous eigenspace 
decompositions.
To get the eigenspace decomposition of a residue matrix or a commuting family of residue matrices
of the equation obtained by applying middle convolutions and additions to an original KZ-type 
equation $\mathcal M$, the necessary data contained in $\Sp\mathcal M$ is depend on the procedure 
of the application, for example, whether $\ker A_j$ is zero or not.
It may be good to check the necessary data for the real calculation, 
simultaneous eigenspace decompositions of families of commuting residue matrices of 
$\mathcal M$
For example, we refer to \cite[Theorem~4.1]{Okz} or \S\ref{sec:fixedpt} when the middle 
convolutions are restricted only on some variables.
\end{remark}

\section{Examples}
\label{sec:Example}
The transformation of $\Sp\mathcal M$ by a middle convolution of $\mathcal M$
is obtained by Theorem~\ref{thm:main}.
Since the transformation is symmetric with respect to the suffices $\{1,\dots,n{-}1\}$
of the variables $(x_0,\dots,x_{n-1})$, we have only to examine the transformation
of the maximal commuting families of residue matrices corresponding to the representatives 
of win types.  In the cases $n=3,\,4,\,5,\,6,\,7,\ldots$, the number of the win types are
%
$W_n=2,\,4,\,9,\,20,\,46,\ldots$, respectively.

In this section, we assume $n=4$ and examine $W_4=4$ cases. 
The results are kept valid by permutations of suffixes $\{1,2,3\}$. 
Note that the total number of maximal commuting families of residue matrices equals 
$K_4=15$.
For simplicity, we assume $\mathcal M$ is homogeneous.
Then $A_{0123}=0$ and $\tilde A_{0123}=\mu$. 

\medskip
{\bf 1}. \ 
\raisebox{-2mm}{\begin{tikzpicture}
\draw
(0,0)--(0,0.2)--(0.2,0.2) (0.3,0.2)--(0.4,0.2)--(0.4,0)
(0.2,0.2)--(0.2,0.4)--(0.4,0.4) (0.5,0.4)--(0.8,0.4)--(0.8,0)
(0.4,0.4)--(0.4,0.6)--(0.6,0.6) (0.7,0.6)--(1.2,0.6)--(1.2,0)
(0.6,0.6)--(0.6,0.8)
node at (0,-0.2){$0$}
node at (0.4,-0.2){$1$}
node at (0.8,-0.2){$2$}
node at (1.2,-0.2){$3$}
;
\end{tikzpicture}
}\qquad
$
\mathcal I=\bigl\{\{0,1\},\{0,1,2\},\{0,1,2,3\}\bigr\}\ 
\xrightarrow{b^0}\ \bigl\{\{1\},\{2\},\{3\}\bigr\}
$
\begin{align*}
U&=\begin{pmatrix}
 1 & 0 & 0\\
 0 & 1 & 0\\
 0 & 0 & 1
\end{pmatrix},\quad V=U^{-1}=\begin{pmatrix}
 1 & 0 & 0\\
 0 & 1 & 0\\
 0 & 0 & 1
\end{pmatrix},\quad \tilde A_* \to V\tilde A_*U\allowdisplaybreaks\\
\tilde A_{01}&=\begin{pmatrix}
  A_{01}{+}\mu & A_{02} & A_{03}\\
 0 & 0 & 0\\
 0 & 0 & 0
\end{pmatrix}\to \begin{pmatrix}
  A_{01}{+}\mu & A_{02} & A_{03}\\
 0 & 0 & 0\\
 0 & 0 & 0
\end{pmatrix}\allowdisplaybreaks\\
\tilde A_{012}&=\begin{pmatrix}
  A_{012}{+}\mu & 0 & A_{03}\\
 0 &  A_{012}{+}\mu & A_{03}\\
 0 & 0 & A_{12}
\end{pmatrix}\to \begin{pmatrix}
  A_{012}{+}\mu & 0 & A_{03}\\
 0 &  A_{012}{+}\mu & A_{03}\\
 0 & 0 & A_{12}
\end{pmatrix}\allowdisplaybreaks\\
[\tilde A_{01}:\tilde A_{012}]&=\{[A_{01}+\mu: A_{012}+\mu],[0: A_{012}+\mu],[0:A_{12}]\}\\
[\tilde A_{01}:\tilde A_{012}]|_{\mathcal K_1}&=[A_{01}+\mu:A_{012}+\mu]|_{\mathrm{ker}\,A_{01}}\\
[\tilde A_{01}:\tilde A_{012}]|_{\mathcal K_2}&=[0:A_{012}+\mu]|_{\mathrm{ker}\,A_{02}}\\
[\tilde A_{01}:\tilde A_{012}]|_{\mathcal K_3}&=[0:A_{12}]|_{\mathrm{ker}\,A_{03}}\\
[\tilde A_{01}:\tilde A_{012}]|_{\mathcal K_\infty}&=[0:A_{12}]|_{\mathrm{ker}\,(A_{0\infty}-\mu)}%
\end{align*}
Here $\tilde A_{ij}$, $U$ and $V$ are block matrices.

\medskip
{\bf 2}. \ 
\raisebox{-2mm}{\begin{tikzpicture}
\draw
(0,0)--(0,0.2)--(0.1,0.2) (0.2,0.2)--(0.4,0.2)--(0.4,0)
(0.2,0.2)--(0.2,0.4)--(0.3,0.4) (0.4,0.4)--(0.8,0.4)--(0.8,0)
(0.4,0.4)--(0.4,0.6)--(0.6,0.6) (0.7,0.6)--(1.2,0.6)--(1.2,0)
(0.6,0.6)--(0.6,0.8)
node at (0,-0.2){$1$}
node at (0.4,-0.2){$2$}
node at (0.8,-0.2){$0$}
node at (1.2,-0.2){$3$}
;
\end{tikzpicture}
}\qquad
$
\mathcal I=\bigl\{\{0,1,2\},\{1,2\},\{0,1,2,3\}\bigr\}\ \xrightarrow{b^0}\ 
\bigl\{\{1,2\},\{1\},\{3\}\bigr\}
$
\begin{align*}
U&=\begin{pmatrix}
 1 & 1 & 0\\
 1 & 0 & 0\\
 0 & 0 & 1
\end{pmatrix},\quad V=U^{-1}=\begin{pmatrix}
 0 & 1 & 0\\
 1 & -1 & 0\\
 0 & 0 & 1
\end{pmatrix},\quad \tilde A_* \to V\tilde A_*U\allowdisplaybreaks\\
\tilde A_{012}&=\begin{pmatrix}
  A_{012}{+}\mu & 0 & A_{03}\\
 0 &  A_{012}{+}\mu & A_{03}\\
 0 & 0 & A_{12}
\end{pmatrix}\to \begin{pmatrix}
  A_{012}{+}\mu & 0 & A_{03}\\
 0 &  A_{012}{+}\mu & 0\\
 0 & 0 & A_{12}
\end{pmatrix}\allowdisplaybreaks\\
\tilde A_{12}&=\begin{pmatrix}
 A_{012}{-}A_{01} & -A_{02} & 0\\
 -A_{01} & A_{012}{-}A_{02} & 0\\
 0 & 0 & A_{12}
\end{pmatrix}\to \begin{pmatrix}
 A_{12} & -A_{01} & 0\\
 0 & A_{012} & 0\\
 0 & 0 & A_{12}
\end{pmatrix}\allowdisplaybreaks\\
[\tilde A_{012}:\tilde A_{12}]&=\{[A_{012}+\mu:A_{12}],[A_{012}+\mu:A_{012}],[A_{12}:A_{12}]\}\\
[\tilde A_{012}:\tilde A_{12}]|_{\mathcal K_1}&=[A_{012}+\mu:A_{012}]|_{\mathrm{ker}\,A_{01}}\\
[\tilde A_{012}:\tilde A_{12}]|_{\mathcal K_2}&=[A_{012}+\mu:A_{012}]|_{\mathrm{ker}\,A_{02}}\\
[\tilde A_{012}:\tilde A_{12}]|_{\mathcal K_3}&=[A_{12}:A_{12}]|_{\mathrm{ker}\,A_{03}}\\
[\tilde A_{012}:\tilde A_{12}]|_{\mathcal K_\infty}&=[A_{12}:A_{12}]|_{\mathrm{ker}\,(A_{0\infty}-\mu)}%
\end{align*}

\begin{tabular}{|c|c|c|c|}
\multicolumn{4}{c}{$A_I^J$ (cf.~Theorem \ref{thm:main})}
\\[1mm]\hline
$J\backslash I$\rule{0mm}{4mm}&$\widetilde{012}$&$\widetilde{12}$&$\widetilde{0123}$\\ \hline
012 & 012+$\mu$&12&$\mu$\\ \hline
12  & 012+$\mu$&012&$\mu$\\ \hline
0123& 12       &12&$\mu$\\ \hline
\end{tabular}\qquad
\begin{tabular}{|c|c|c|c|}\hline
$j\backslash I$\rule{0mm}{4mm}&$\widetilde{012}$&$\widetilde{12}$&$\widetilde{0123}$\\ \hline
1    & 012+$\mu$&012&$\mu$\\ \hline
2    & 012+$\mu$&012&$\mu$\\ \hline
3    & 12  &12&$\mu$\\ \hline
$\infty$& 12&12&$\mu$\\ \hline
\end{tabular}
\begin{remark}\label{rem:main}
The simultaneous eigenspace decomposition of $\bigl(\tilde A_I\bigr)_{I\in\mathcal I}$ is
obtained by $A_I^J$ in Theorem~\ref{thm:main}.
The above left table is the $(n{-}1)\times(n{-}1)$ matrix 
whose $(J,I)$-element with $I\in\mathcal I$ and $J\in\mathcal I$ is the suffix $K$ of $A_K=A_I^J$.
Moreover $012+\mu$ in the table means $A_{012}+\mu$. 
Then $K$ contains $0$ if and only if $I\supset J$ and the term ``$+\mu$" exists
if and only if $0\in I\supset J$.
Here the last column corresponding to $\tilde A_{0123}$ is omitted.
Similarly, the above right table shows $A_I^{\{j\}}$ and $A_I^{L_n}$
which describe $\bigl(\tilde A_I|_{\mathcal K_j}\bigr)_{I\in\mathcal I}$
and $\bigl(\tilde A_I|_{\mathcal K_\infty}\bigr)_{I\in\mathcal I}$, respectively.
\end{remark}

\medskip
{\bf 3}. \ 
\raisebox{-2mm}{\begin{tikzpicture}
\draw
(0,0)--(0,0.2)--(0.1,0.2) (0.2,0.2)--(0.4,0.2)--(0.4,0)
(0.2,0.2)--(0.2,0.4)--(0.3,0.4) (0.4,0.4)--(0.8,0.4)--(0.8,0)
(0.4,0.4)--(0.4,0.6)--(0.5,0.6) (0.6,0.6)--(1.2,0.6)--(1.2,0)
(0.6,0.6)--(0.6,0.8)
node at (0,-0.2){$1$}
node at (0.4,-0.2){$2$}
node at (0.8,-0.2){$3$}
node at (1.2,-0.2){$0$}
;
\end{tikzpicture}
}\quad
$
\mathcal I=\bigl\{\{0,1,2,3\},\{1,2,3\},\{1,2\}\bigr\}\ \xrightarrow{b^0}\ 
\bigl\{\{1,2,3\},\{1,2\},\{1\}\bigr\}
$
\begin{align*}
U&=\begin{pmatrix}
 1 & 1 & 1\\
 1 & 1 & 0\\
 1 & 0 & 0
\end{pmatrix},\quad V=U^{-1}=\begin{pmatrix}
 0 & 0 & 1\\
 0 & 1 & -1\\
 1 & -1 & 0
\end{pmatrix},\quad \tilde A_* \to V\tilde A_*U\allowdisplaybreaks\\
\tilde A_{123}&=\begin{pmatrix}
 -A_{01} & -A_{02} & -A_{03}\\
 -A_{01} & -A_{02} & -A_{03}\\
 -A_{01} & -A_{02} & -A_{03}
\end{pmatrix}\to \begin{pmatrix}
 A_{123} & -A_{01}-A_{02} & -A_{01}\\
 0 & 0 & 0\\
 0 & 0 & 0
\end{pmatrix}\allowdisplaybreaks\\
\tilde A_{12}&=\begin{pmatrix}
 A_{012}{-}A_{01} & -A_{02} & 0\\
 -A_{01} & A_{012}{-}A_{02} & 0\\
 0 & 0 & A_{12}
\end{pmatrix}\to \begin{pmatrix}
 A_{12} & 0 & 0\\
 0 & A_{12} & -A_{01}\\
 0 & 0 & A_{012}
\end{pmatrix}\allowdisplaybreaks\\
[\tilde A_{123}:\tilde A_{12}]&=\{[A_{123}:A_{12}],[0:A_{12}],[0:A_{012}]\}\\
[\tilde A_{123}:\tilde A_{12}]|_{\mathcal K_1}&=[0:A_{012}]|_{\mathrm{ker}\,A_{01}}\\
[\tilde A_{123}:\tilde A_{12}]|_{\mathcal K_2}&=[0:A_{012}]|_{\mathrm{ker}\,A_{02}}\\
[\tilde A_{123}:\tilde A_{12}]|_{\mathcal K_3}&=[0:A_{12}]|_{\mathrm{ker}\,A_{03}}\\
[\tilde A_{123}:\tilde A_{12}]|_{\mathcal K_\infty}&=[A_{123}:A_{12}]|_{\mathrm{ker}\,(A_{0\infty}-\mu)}%
\end{align*}

\medskip
{\bf 4}. \ 
\raisebox{-2mm}{\begin{tikzpicture}
\draw
(0,0)--(0,0.2)--(0.2,0.2) (0.3,0.2)--(0.4,0.2)--(0.4,0)
(0.8,0)--(0.8,0.2)--(0.9,0.2) (1,0.2)--(1.2,0.2)--(1.2,0)
(0.2,0.2)--(0.2,0.4)--(0.6,0.4) (0.7,0.4)--(1,0.4)--(1,0.2)
(0.6,0.4)--(0.6,0.6)
node at (0,-0.2){$0$}
node at (0.4,-0.2){$1$}
node at (0.8,-0.2){$2$}
node at (1.2,-0.2){$3$}
;
\end{tikzpicture}
}\qquad
$\mathcal I=\bigl\{\{0,1\},\{0,1,2,3\},\{2,3\}\bigr\}\ \xrightarrow{b^0}\ 
\bigl\{\{1\},\{2,3\},\{2\}\bigr\}$
\begin{align*}
U&=\begin{pmatrix}
 1 & 0 & 0\\
 0 & 1 & 1\\
 0 & 1 & 0
\end{pmatrix},\quad V=U^{-1}=\begin{pmatrix}
 1 & 0 & 0\\
 0 & 0 & 1\\
 0 & 1 & -1
\end{pmatrix},\quad \tilde A_* \to V\tilde A_*U\allowdisplaybreaks\\
\tilde A_{01}&=\begin{pmatrix}
  A_{01}{+}\mu & A_{02} & A_{03}\\
 0 & 0 & 0\\
 0 & 0 & 0
\end{pmatrix}\to \begin{pmatrix}
  A_{01}{+}\mu & A_{03}+A_{02} & A_{02}\\
 0 & 0 & 0\\
 0 & 0 & 0
\end{pmatrix}\allowdisplaybreaks\\
\tilde A_{23}&=\begin{pmatrix}
 A_{23} & 0 & 0\\
 0 & A_{023}{-}A_{02} & -A_{03}\\
 0 & -A_{02} & A_{023}{-}A_{03}
\end{pmatrix}\to \begin{pmatrix}
 A_{23} & 0 & 0\\
 0 & A_{23} & -A_{02}\\
 0 & 0 & A_{023}
\end{pmatrix}\allowdisplaybreaks\\
[\tilde A_{01}:\tilde A_{23}]&=\{[\!A_{01}+\mu:A_{23}],[0:A_{23}],[0:A_{023}]\}
\allowdisplaybreaks\\
[\tilde A_{01}:\tilde A_{23}]|_{\mathcal K_1}&=[A_{01}+\mu:A_{23}]|_{\mathrm{ker}\,A_{01}}
\allowdisplaybreaks\\
[\tilde A_{01}:\tilde A_{23}]|_{\mathcal K_2}&=[0:A_{023}]|_{\mathrm{ker}\,A_{02}}
\allowdisplaybreaks\\
[\tilde A_{01}:\tilde A_{23}]|_{\mathcal K_3}&=[0:A_{023}]|_{\mathrm{ker}\,A_{03}}
\allowdisplaybreaks\\
[\tilde A_{01}:\tilde A_{23}]|_{\mathcal K_\infty}&=[0:A_{23}]|_{\mathrm{ker}\,(A_{0\infty}-\mu)}%
\end{align*}

\begin{tabular}{|c|c|c|}
\multicolumn{3}{c}{Case {\bf 1}}
\\\hline
$\widetilde A$\rule{0mm}{4mm} &$\widetilde{01}$&$\widetilde{012}$\\ \hline
01&01+$\mu$&012+$\mu$\\ \hline
012&0&012+$\mu$\\ \hline
0123&0&12\\ \hline
1&01+$\mu$&012+$\mu$\\ \hline
2&0&012+$\mu$\\ \hline
3&0&12\\ \hline
$\infty$&0&12\\ \hline
\end{tabular}
\quad
\begin{tabular}{|c|c|c|}
\multicolumn{3}{c}{Case {\bf 3}}
\\\hline
$\widetilde A$\rule{0mm}{4mm}&$\widetilde{123}$&$\widetilde{12}$\\ \hline
0123&123&12\\ \hline
123&0&12\\ \hline
12&0&012\\ \hline
1&0&012\\ \hline
2&0&012\\ \hline
3&0&12\\ \hline
$\infty$&123&12\\ \hline
\end{tabular}
\quad
\begin{tabular}{|c|c|c|}
\multicolumn{3}{c}{Case {\bf 4}}
\\\hline
$\widetilde A$\rule{0mm}{4mm} &$\widetilde{01}$&$\widetilde{23}$\\ \hline
01&01+$\mu$&23\\ \hline
0123&0&23\\ \hline
23&0&023\\ \hline
1&01+$\mu$&23\\ \hline
2&0&023\\ \hline
3&0&023\\ \hline
$\infty$&0&23\\ \hline
\end{tabular}

\medskip
We show a computer program displaying the result in this section which uses functions
in a library \cite{Or} of the computer algebra \texttt{Risa/Asir}. 
Then the result including the figures of tournaments is displayed through a PDF file created by \TeX.

\medskip
\begin{verbatim}
 N=4;    /* N-2=2 variables HG */
 T=os_md.symtournament(N|to="T");        /* T: types */
 for(S="";T!=[];T=cdr(T)){               /* R: win types */
   R=os_md.xytournament(car(T),0|verb=21,winner="all");
   for(;R!=[];R=cdr(R)){
     C=car(R);
     S=S+"\\raisebox{-2mm}{"             /* S: source of TeX */
     +os_md.xytournament(C[0],0|teams=C[1],winner=0) /* figure */
     +"}\\qquad"+rtostr(C[4])+"\\ $\\to$ \\ "+rtostr(C[3])
     +os_md.midKZ(C[2],C[3]);            /* spectra */
   }
 }
 os_md.dviout(S);                        /* display */
\end{verbatim}
\medskip

In the above program
\begin{itemize}
\item
In the first line, the number of variables $n=4$ is given by \texttt{N=4}.
\item
In the 2-nd line, all the types \texttt{T} for \texttt{N} teams are obtained.
\item
In the 4-th line, all win types $\mathcal I$ are obtained in \texttt{R}.
\item
From 5-th line, $\Sp(\mc_{x_0,\mu}\mathcal M)$ and the findings as presented in this section 
are transformed into a source text \texttt{S} in \TeX\ and in the last line it is displayed
using a PDF file transformed from the source text.
\end{itemize}
\begin{remark}
The top insertion imbeds the tournaments of $n{-}1$ teams in those of $n$ teams.
The image of this imbedding is the tournaments of $n$ teams with $b^0(L_n)=\{n{-}1\}$.
This corresponds to the KZ-type equation $\mathcal M$ with $A_{i,n{-}1}=0$ \ $(0\le i\le n-2)$.
Hence our result of KZ-type equations with $n{-}1$ variables follows from that of KZ-type equations with 
$n$ variables. 

The first two examples in this section correspond to this imbedding  
and the results for $n=3$ are obtained by omitting  $b^0(\{0,1,2,3\})=\{3\}$. 
Namely, we get them by the first $2\times 2$ blocks of the matrices in these examples.
Moreover we omit the term $\mathcal K_3$ and the last terms of the simultaneous eigenspace 
decompositions. The term $A_{012}$ can remain.
\end{remark}


\section{Further considerations}
\label{sec:further}
\subsection{Infinite point}
\label{sec:infinity}
KZ-type equation $\mathcal M$ in \S\ref{sec:KZ} is considered to be defined on the configuration
space of $n{+}1$ points of $\mathbb P^1$.
By a linear fractional transformation transforming the infinite point to a finite point, 
we have a KZ-type equation with $n{+}1$ variables which has no singularity at infinite point.
Then all the singular points are finite points and it may be easier to understand a symmetry
among singular points.
If the original equation has $n{-}1$ variables, 
the resulting KZ-type equation has $n$ variables and is characterized by the condition
\begin{equation}\label{eq:regularinf}
 A_{i\infty}:=\sum_{\nu=0}^{n-1}A_{i\nu}=0\quad(0\le i<n)
\end{equation}
on the residue matrices $A_{ij}$.
Hence we assume the following.
\begin{definition}
$\infty$ is a {\bf pseudo-singular point}, namely,
there exist $\mu_i\in\mathbb C$ such that
\begin{equation}\label{eq:quasiregularinf}
 A_{i\infty}=\mu_i\quad(0\le i<n)
.
\end{equation}
Here $\mu_i$ mean scalar matrices.
\end{definition}

If we apply 
$\Ad\bigl((x_0-x_1)^\lambda(x_0-x_2)^{\lambda}(x_1-x_2)^{-\lambda}\bigr)$ to
$\mathcal M$,
$A_{0\infty}$ is changed into $A_{0\infty}-\lambda$ and $A_{i\infty}$ for $i\ne0$ are unchanged.
Hence the KZ-tye equation with a pseudo-singular infinite point
can be changed to a equation satisfying \eqref{eq:regularinf}.
We examine the middle convolution of $\tilde A_{i\infty}$.
Note that
\begin{align*}
\tilde A_{0\infty}&=-\sum_{\nu=1}^{n-1}A_{0\nu}
   =\begin{pmatrix}
    -A_{01}-\mu&-A_{02}&\cdots&-A_{0,n-1}\\
    -A_{01}&-A_{02}-\mu&\cdots&-A_{0,n-1}\\[-1.5mm]
    \vdots&\vdots&\ddots&\vdots\\
    -A_{01}&-A_{02}&\cdots&-A_{0,n-1}-\mu
     \end{pmatrix},\allowdisplaybreaks\\
\tilde A_{1\infty}&=-\sum_{\nu=0}^{n-1}A_{1\nu}
   =\begin{pmatrix}
    A_{0\infty}+A_{1\infty}+A_{01}-\mu & 0&\cdots&0\\
    A_{01}&A_{1\infty}&\cdots&0\\[-1.5mm]
    \vdots&\vdots&\ddots&\vdots\\
    A_{01}&0&\cdots&A_{1\infty}
     \end{pmatrix}
.
\end{align*}
Hence for the variable $x_0$, 
the middle convolution of the KZ-type equation $\mathcal M$
satisfying \eqref{eq:quasiregularinf} 
is defined by $\mc_{x_0,\mu_0}$. 
Then
\begin{align*}
  \widetilde A_{0\infty}(v)_{L_n^0}&=\bigl((A_{0\infty}-\mu)v\bigr)_{L_n^0},&
  \widetilde A_{1\infty}(v)_{L_n^0}&=\bigl((A_{01}+A_{1\infty})v\bigr)_{L_n^0},\\
  \widetilde A_{0\infty}&=0\mod \mathcal K_\infty,&
  \widetilde A_{i\infty}&=\mu_i\mod \mathcal K_i\quad(1\le i<n),
\end{align*}
and therefore
\begin{align}\label{eq:mcquasi}
  \mathcal K_\infty=V_{L_n^0},\ \ {\overline A}_{0\infty}=0,\ \ 
  {\overline A}_{i\infty}=\mu_i\quad(0<i<n)
\end{align}
and $\mc_{x_0,\mu_0}\mathcal M$ also has a pseudo-singular infinite point.

\subsection
{Fixed singular points}
\label{sec:fixedpt}
We examine the KZ-type equation
\begin{align}\label{eq:relKZ}
\mathcal M : 
 \frac{\p u}{\p x_i}
  &=\sum_{\substack{0\le\nu\le n-1\\ \nu\ne i}}
    \frac{A_{i\nu}}{x_i-x_\nu}u+\sum_{q=1}^m\frac{B_{iq}}{x_i-y_q}u
 \quad(i=0,\dots,n-1)
\end{align}
which have fixed singular points $y_1,\ldots,y_m$ together with $x_i=x_j$.

We may assume $\mathcal M$ has a pseudo-singular infinite point without loss of generality.

For $\{i_1,\dots,i_p\}\subset \{0,\dots,n-1\}$ and $\{j_1,\dots,j_q\}\subset \{n,\dots,m+n-1\}$,
put
\begin{align*}
A_{i_1,\dots,i_p;j_1,\dots,j_q}&:=
  \sum_{1\le \nu<\nu'\le p}A_{i_\nu i_{\nu'}} 
   + \sum_{\substack{1\le \nu \le p\\ 1\le \nu'\le q}}B_{i_\nu j_{\nu'}-n+1}
.
\end{align*}
We may think that we put $y_j=x_{n-1+j}$ for $j=1,\dots,m$.  Note that
\begin{align*}
   A_{i_1,\dots,i_p;j_1,\dots,j_q}=A_{i_1,\dots,i_p,j_1,\ldots,j_q}-A_{j_1,\ldots,j_q}
.
\end{align*}
Here the terms $A_{j_\nu j_{\nu'}}$ in the above right hand side are cancelled.

The integrability condition of $\mathcal M$ is
\begin{equation}
 \begin{split}
 [A_{ij},A_{k\ell}]&=[A_{i;q},A_{j;q'}]=[A_{ij},A_{k;q}]=0
,\\
 [A_{ij},A_{ijk}]&=[A_{ij},A_{ij;q}]=[A_{i;q},A_{ij;q}]=0.
 \end{split}
\end{equation}
Here $i,\,j,\,k,\,\ell\in L_n$ and $q,\,q'\in\{n,n{+}1,\ldots,n{+}m{}1\}$ are distinct numbers.
Since $[A_{01},A_{01\cdots k;q}]=[A_{01},\sum\limits_{0\le i<j\le k}A_{i,j}+A_{01;q}+\sum\limits_{i=2}^k A_{i;q}]=0$ etc., we have
\begin{equation}
 \begin{split}
 [A_I,A_J]&=0\quad(I\cap J=\emptyset\text{ or }I\subset J\text{ or }I\supset J),\\
 [A_{I},A_{J;q}]&=0\quad(I\cap J=\emptyset\text{ or }I\subset J),\\ 
 [A_{I;q},A_{J;q'}]&=0\quad(I\cap J=\emptyset\text{ and }q\ne q')
 \end{split}
\end{equation}
for $I,\,J\subset\{0,1,\dots,n-1\}$ and $\{q,q'\}\subset\{n,n+1,\dots,n+m-1\}$.

Hereafter in \S\ref{sec:fixedpt}, we assume \eqref{eq:quasiregularinf}.
\begin{definition}
For a pair of finite sets $(L,L')$ with $L\cap L'=\emptyset$, 
a maximal commuting $\mathcal I$ of $(L,L')$ is defined as follows.
Namely, $\mathcal I=\{\mathcal I^{(\nu)}\mid \nu\in L'\}$, 
\begin{equation}
  L_n=\bigsqcup_{j\in L'} S_j
\end{equation}
and $\mathcal I^{(j)}$ are maximal commuting family of $S_j\cup\{j\}$, respectively.

Moreover putting $L'=\{r_1,\dots,r_m\}$, we define
\begin{equation}
\widehat{\mathcal I}:=\mathcal I\cup \bigcup_{j=2}^m\{\widehat S_j\},\  
  \widehat S_j:=\bigcup_{\nu=1}^j\bigl(S_\nu\cup\{\nu\}\bigr).
\end{equation}
\end{definition}
Here $n$ are numbers of the variables and the fixed points, respectively, and
$n=|S_1|+\cdots+|S_m|$.

We can consider the middle convolution of $\mathcal M$ with respect to any one of the
variables $x_0,\ldots,x_{n-1}$ and the resulting change of the eigenspace decompositions
of residue matrices are obtained from Theorem~\ref{thm:main}.
We consider residue matrices in \eqref{eq:mcquasi} and those corresponding to $\mathcal I$
and define the base of the matrices by $\widehat{\mathcal I}$.
In particular, if $n=1$, this coincides with the result given in \cite{DR}. 

\begin{remark}
The tournament corresponds to $\widehat{\mathcal I}$ is characterized as follows.
The team $(n+m-j)$ and the team $(n+m-1-j)$ will face in a match
$j{-}1$ games before the final if they have won the preceding matches 
($j=1,\dots,m$).
\end{remark}

\begin{exmp}
The representatives of the maximal commuting family of $(L,L')=(\{0,1,2\},\{3,4,5\})$ 
under the permutations of the elements of $L$ and those of $L'$ are
\begin{align*}
1+1+1&:\bigl\{\{0,3\},\{1,4\},\{2,5\}\bigr\},\\
2+1+0&:\bigl\{\{0,3\},\{1,4\},\{0,2,3\}\bigr\},\ 
\bigl\{\{0,1\},\{2,3\},\{0,1,4\}\bigr\},\\
3+0+0&:\bigl\{\{0,1\},\{2,3\},\{0,1,2,3\}\bigr\},\ 
\bigl\{\{0,1\},\{0,1,2\},\{0,1,2,3\}\bigr\},\\
&\quad\bigl\{\{0,1\},\{0,1,3\},\{0,1,2,3\}\bigr\},\ 
\bigl\{\{0,3\},\{0,1,3\},\{0,1,2,3\}\bigr\}.
\end{align*}
When $S_3=\{0,1,2\}$ and  $S_4=S_5=\emptyset$, we have 4 families indicated by 
$3+0+0\ (=|S_1|+|S_2|+|S_3|)$ in the above because $W_{3+1}=4$.
The numbers of maximal commuting families are $3K_4=45$, 
$3!\cdot3K_3=54$, $6$ according to the cases $3+0+0$, $2+1+0$, $1+1+1$, respectively.
The total number of them equals 105.

%
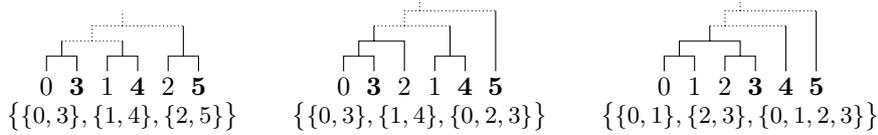
\begin{tikzpicture}
\draw[densely dotted]
(1,0.6)--(1,0.8)
(0.2,0.4)--(1,0.4)
(0.6,0.4)--(0.6,0.6)--(1.8,0.6)
;
\draw
(0,0)--(0,0.2)--(0.4,0.2)--(0.4,0)
(0.8,0)--(0.8,0.2)--(1.2,0.2)--(1.2,0)
(1.6,0)--(1.6,0.2)--(2,0.2)--(2,0)
(0.2,0.2)--(0.2,0.4) (1,0.4)--(1,0.2)
(1.8,0.6)--(1.8,0.2)
node at (0,-0.2){$0$}
node at (0.4,-0.2){$\mathbf 3$}
node at (0.8,-0.2){$1$}
node at (1.2,-0.2){$\mathbf 4$}
node at (1.6,-0.2){$2$}
node at (2,-0.2){$\mathbf 5$}
node at (1,-0.6) {\small$\bigl\{\{0,3\},\{1,4\},\{2,5\}\bigr\}$}
;
\end{tikzpicture}
\quad
\begin{tikzpicture}
\draw[densely dotted]
(0.4,0.6)--(1.4,0.6)
(0.8,0.6)--(0.8,0.8)--(2,0.8)
(1,0.8)--(1,1)
;
\draw
(0,0)--(0,0.2)--(0.4,0.2)--(0.4,0)
(1.2,0)--(1.2,0.2)--(1.6,0.2)--(1.6,0)
(0.2,0.2)--(0.2,0.4)--(0.8,0.4)--(0.8,0)
(0.4,0.4)--(0.4,0.6)  (1.4,0.6)--(1.4,0.2)
 (2,0.8)--(2,0)
node at (0,-0.2){$0$}
node at (0.4,-0.2){$\mathbf 3$}
node at (0.8,-0.2){$2$}
node at (1.2,-0.2){$1$}
node at (1.6,-0.2){$\mathbf 4$}
node at (2,-0.2){$\mathbf 5$}
node at (1,-0.6) {\small$\bigl\{\{0,3\},\{1,4\},\{0,2,3\}\bigr\}$}
;
\end{tikzpicture}
\quad
\begin{tikzpicture}
\draw[densely dotted]
(0.6,0.6)--(1.6,0.6)
(0.8,0.6)--(0.8,0.8)--(2,0.8)
(1,0.8)--(1,1)
;
\draw
(0,0)--(0,0.2)--(0.4,0.2)--(0.4,0)
(0.8,0)--(0.8,0.2)--(1.2,0.2)--(1.2,0)
(0.2,0.2)--(0.2,0.4)--(1,0.4)--(1,0.2)
(0.6,0.4)--(0.6,0.6)  (1.6,0.6)--(1.6,0)
(2,0.8)--(2,0)
node at (0,-0.2){$0$}
node at (0.4,-0.2){$1$}
node at (0.8,-0.2){$2$}
node at (1.2,-0.2){$\mathbf 3$}
node at (1.6,-0.2){$\mathbf 4$}
node at (2,-0.2){$\mathbf 5$}
node at (1,-0.6) {\small$\bigl\{\{0,1\},\{2,3\},\{0,1,2,3\}\bigr\}$}
;
\end{tikzpicture}

\end{exmp}
\begin{exmp}\label{ex:ODE}
Suppose $n=1$.
Then there are $m$ maximal commuting families of $(L,L')$.
For example, when $m=4$, we have\\[-2mm]
\raisebox{2mm}{$\bigl\{\mathcal I\bigr\}=\bigl\{\bigl\{\{0,1\}\bigr\},\bigl\{\{0,2\}\bigr\},
  \bigl\{\{0,3\}\bigr\},\bigl\{\{0,4\}\bigr\}\bigr\}$,\quad
$\widehat{\mathcal I}=\widehat{\{\{0,3\}\}}\ :$}
\begin{tikzpicture}
\draw[densely dotted]
(0,0.2)--(0.4,0.2)
(0.2,0.4)--(1,0.4)
(0.6,0.6)--(1.6,0.6);
\draw
(0.8,0.2)--(1.2,0.2)
(0,0)--(0,0.2)
(0.4,0.2)--(0.4,0)
(0.8,0)--(0.8,0.2)
(1.2,0.2)--(1.2,0)
(0.2,0.2)--(0.2,0.4)
(1,0.4)--(1,0.2)
(0.6,0.4)--(0.6,0.6)
(0.6,0.4)--(0.6,0.6)
(1.6,0.6)--(1.6,0)
(1.1,0.6)--(1.1,0.8)
node at (0,-0.2){$\mathbf 1$}
node at (0.4,-0.2){$\mathbf 2$}
node at (0.8,-0.2){$0$}
node at (1.2,-0.2){$\mathbf 3$}
node at (1.6,-0.2){$\mathbf 4$}
;
\end{tikzpicture}
\end{exmp}
\begin{remark}
Suppose $m=1$.
The integrability condition of the KZ-type equation $\mathcal M$ equals to 
that of the extended KZ-type equation with $n{+}1$ variables 
which are imposed on residue matrices  $A_{ij}$ ($0\le i<j\le n$). 
Hence the KZ-type equation adding the partial derivation with respect to the variable 
$x_n=y_1$ satisfies the integrability condition. On the other hand, the product of
a solution to the extended equation and any function of $x_n$ satisfies the original equation.
Hence the dimension of the solutions to the original equation with variables $(x_0,\dots,x_n)$ 
is infinite. 

\end{remark}

\subsection
{Spectra and Accessory parameters}
\label{sec:accessory}
Suppose $n=1$ and $m=2$. Then we have the Fuchsian ststem
\begin{equation}\label{eq:ODE3}
  \mathcal N: \frac{du}{dx_0}=\frac{A_{01}}{x_0-x_1}u+\frac{A_{02}}{x_0-x_2}u
\end{equation}
with three singular points $x_1$, $x_2$ and $\infty$.  Putting
\begin{equation}\label{eq:KZ3}
 A_{12}=-A_{01}-A_{02}
,
\end{equation}
namely, $A_{12}=A_{0,\infty}$, the equation $\mathcal N$ is extended to
the KZ-type equation $\mathcal M$ with the three variables $(x_0,x_1,x_2)$.
Conversely, we may assume \eqref{eq:KZ3} for 
the irreducible KZ-type equation with $n=3$ (cf.~Definition~\ref{def:homog}). 

In general, $\mathcal N$ is not necessarily rigid and has $2r$ accessory parameters 
($r=0,1,\ldots$). 
For example, $A_{01}$ and  $A_{02}$ are generic matrices in $M(3,\mathbb C)$, 
$\mathcal N$ has 2 accessory parameters.
Since additions and middle convolutions are invertible transformations
which dos not change the number of accessory parameters, we get KZ-type equations with 
$n$ variables which have $2r$ accessory parameters ($n=3,4,\ldots$).
Conversely, we know the number of accessory parameters of a KZ-type equation
if it can be transformed into a KZ-type equation with three variables.

We define that a KZ-type equation $\mathcal M$ is {\bf rigid} if 
the equation is uniquely determined by $\Sp\mathcal M$ with no accessory parameter.
It is an interesting problem to examine that 
a KZ-type equation is transformed into another KZ-type equation
by a successive application of the transformations we have considered.
In particular, the problem is quite interesting if the equations are rigid
and so is the problem determining irreducible KZ-type equations which cannot be 
reduced the rank by any application of these transformations. 

A KZ-type equation $\mathcal M$ obtained by applying these applications to 
the trivial equation $u'=0$ is rigid. 
In this case, owing to symmetries of $\Sp\mathcal M$, we may get several relations between 
the solutions to the equation as in the case of Kummer's relation for 
Gauss hypergeometric functions (cf.~\cite[Remark~5.17]{MO2024}).

In general, for a holonomic system $\mathcal M$, we blow up its singular locus to normal crossing 
singular points and get commuting residue matrices attached to normal crossing divisors 
(cf.~\cite{KO}). 
The {\bf spectra} $\Sp\mathcal M$ is the set of conjugacy classes of commuting residue matrices at 
the normal crossing singular points.  When the commuting matrices are semisimple, the conjugacy 
class are the set of simultaneous eigenvalues and their multiplicities.
\subsection
{Semilocal monodromy}
\label{sec:semilocal}
For example, \cite{Osemilocal} calculates  $[\widetilde A_{03}+\widetilde A_{04}]$ from 
$[A_{0i}]$ and $[A_{03}+A_{04}]$ etc.\ in the case $L=\{0\}$ and $L'=\{1,2,3,4\}$.
Here we examine to calculate $[\widetilde A_{03}+\widetilde A_{04}]$ in the case $L=\{0, 1, 2\}$
and $L'=\{3, 4, 5\}$. 
Formally we have $A_{03}+A_{04}=A_{034}-A_{34}$ and we may calculate as if $A_{34}$ etc.\ 
exist. 

This corresponds to the tournaments such that the teams (3) and (4) play the final game
and so do the teams (4) and (5) if they have won the former games. 
Moreover we may restrict to the tournaments such that teams (3) and (4) play 
a semi-final  or quarter final game if they have won the preceding games.

For $\{0,1,2\}\supset I_1\supset I_2$, 
$\widetilde A_{I_134}-\widetilde A_{I_234}$ is expressed by residue matrices in 
\eqref{eq:relKZ}.
To get the eigenspace decomposition of $\widetilde A_{I_134}-\widetilde A_{I_234}$ from the residue matrices $A_{ij}$ of \eqref{eq:relKZ}, we examine the tournaments containing the games corresponding to $I_134$ and $I_234$ and apply Theorem~\ref{thm:main} to maximal commuting residue matrices
expressed by the tournaments.  Then the simultaneous eigenspace decompositions
containing the residue matrices $A_{I'_134}$ and $A_{I'_234}$ with $I'_1\supset I'_2$
appear but we have only to calculate the eigenspace decompositions containing the 
matrices $A_{I_1'34}-A_{I'_234}$.
In this case, we may have $I'_1=I'_2$ even if $I_1\supsetneqq I_2$.

\subsection{single-elimination tournaments}
The following table shows the relation between KZ-type equations and single-elimination tournaments
discussed in this paper.

\noindent
\scalebox{0.98}{\begin{tabular}{|c|c|}\hline
KZ-type equation with $n$ variables&Tournament of $n$ teams\\\hline\hline
\scalebox{0.9}[1]{Family of maximal commuting residue matrices}&Tournament\\\hline
Spectra of a KZ-type equation&Set of all tournaments\\\hline
Singular points &Semi-final matches \\\hline
Local coordinate for desingularization &
\scalebox{0.9}[1]{Result of a tournament before semi-final}\\\hline
Variable of middle convolution&Winner of tournament\\\hline
Base of upper triangulation of the family&Result of all matches
\\\hline
Middle convolution&Deletion and insertion of the winner\\\hline
Kernels to define middle convolution&Basic/Top insertion of the winner\\\hline
With other $m$ fixed singular points&Divide $n$ teams into $m$ groups\\\hline
\end{tabular}}

\end{document}